\documentclass[11pt]{article}
\usepackage{amssymb,amsmath,amsthm,bm,mathrsfs}
\usepackage{url,color,units,graphicx}
\usepackage[margin=1in]{geometry}
\usepackage[labelsep=period,labelfont={bf,small},textfont={small},margin=0.8in]{caption}

\DeclareMathOperator{\Cov}{Cov}

\newcommand{\<}{\langle}
\renewcommand{\>}{\rangle}

\newcommand{\Z}{\mathbb{Z}}

\newcommand{\R}{\mathbb{R}}

\newcommand{\cU}{\mathscr{U}}
\newcommand{\cF}{\mathscr{F}}

\newcommand{\cT}{\mathcal{T}}
\newcommand{\cI}{\mathcal{I}}

\newcommand{\sD}{\mathscr{D}}
\newcommand{\cD}{\mathcal{D}}

\newcommand{\lip}{\text{\rm Lip}}

\renewcommand{\P}{\mathrm{P}}
\newcommand{\E}{\mathrm{E}}

\newcommand{\1}{\boldsymbol{1}}
\renewcommand{\d}{{\rm d}}

\newcommand{\e}{{\rm e}}
\renewcommand{\geq}{\geqslant}
\renewcommand{\leq}{\leqslant}
\renewcommand{\ge}{\geqslant}
\renewcommand{\le}{\leqslant}

\author{Davar Khoshnevisan\\University of Utah
	\and Kunwoo Kim\\POSTECH
	\and Carl Mueller\\University of Rochester
	\and Shang-Yuan Shiu\\National Central University}
\title{\bf Dissipation in parabolic SPDEs\thanks{
	Research supported in part by the NSF grant DMS-1307470 [D.K.], 
	the NRF grant 2017R1C1B1005436 and the TJ Park Science Fellowship [K.K],
        a Simons grant [C.M.], and MOST grant MOST103-2628-M-008-003-MY4 
         [S.-Y. S.].
	Parts of this research were funded by the NSF grant DMS-1440140
	while three of the authors [D.K., K.K., and C.M.]
	were in residence at the Mathematical Sciences Research
	Institute at UC Berkeley in Fall of 2015. Three of us
 	[D.K., K.K., and C.M.] wish to also thank the Banff International 
	Research Center for financial support under the Research in Teams Program.  
}}
\date{Version: March  19, 2019}
\newtheorem{stat}{Statement}[section]
\newtheorem{proposition}[stat]{Proposition}
\newtheorem*{prop}{Proposition}

\newtheorem{theorem}[stat]{Theorem}
\newtheorem{lemma}[stat]{Lemma}
\theoremstyle{definition}

\newtheorem{remark}[stat]{Remark}

\numberwithin{equation}{section}

\begin{document}
\maketitle
\begin{abstract} 
The study of intermittency for the parabolic Anderson problem usually 
focuses on the moments of the solution which can describe the high peaks  in the probability space. In this paper we set up the equation on a finite spatial interval, and study the other part of intermittency, i.e., the part of the probability space on which the solution is close to zero.  This set has probability very close to one, and we show that on this set, the supremum of the solution over space is close to 0.  As a consequence, we find that almost surely the spatial supremum of the solution tends to zero exponentially fast as time increases. We also show that if the noise term is very large, then the probability of the set on which the supremum of the solution is very small has a very high probability.  \\[.1in]


\end{abstract}


\section{Introduction, background, and main results}

Consider the solution $u$ to the parabolic stochastic PDE (SPDE, for short),
\begin{equation}\label{SHE}
	\partial_t u =  \partial^2_x u + \sigma(u)\xi
\end{equation}
where $u=u(t\,,x)$, $t>0$,  $x$ lies in the torus $\mathbb{T}:=[-1\,,1]$,
$\sigma:\R\to\R$ is non-random and Lipschitz continuous, and 
$\xi=\xi(t\,,x)$ denotes space-time white noise. The initial 
profile $u_0(x):=u(0,x)$ is assumed to be non-random, and to satisfy
\begin{equation}\label{u_0>0}
	0<\inf_{x\in\mathbb{T}} u_0(x) \le\sup_{x\in\mathbb{T}}u_0(x) <\infty.
\end{equation}
The  Laplace operator $\partial^2_x$ in \eqref{SHE} is endowed with
periodic boundary conditions on $\mathbb{T}$.  

According to the standard theory of SPDEs,  there exists a unique almost 
surely continuous random field $u$ that satisfies
\[
	\sup_{t\in(0,T)}\sup_{x\in\mathbb{T}}
	\E\left(|u(t\,,x)|^k\right)<\infty
	\qquad\text{for all $T>0$ and $k\ge2$},
\]
that solves \eqref{SHE}; see \cite{Dalang,cbms,Walsh}. See also \S\ref{sec:mild} below
for further details. 

In addition, we suppose
that there exist two real numbers $\lip_\sigma\ge{\rm L}_\sigma>0$ such that
\begin{equation}\label{cone}
	{\rm L}_\sigma\le \left| \frac{\sigma(a)}{a}\right|\le \lip_\sigma
	\qquad\text{for every }a\in\R\setminus\{0\}.\footnote{Clearly, 
	\eqref{cone} is equivalent to the
	condition that $\sigma(0) = 0$ and $\inf_{a\neq0}|\sigma(a)/a|>0.$
	Therefore, we can always choose $\lip_\sigma$ to be the Lipschitz
	constant of $\sigma$. This remark also justifies 	the notation for $\lip_\sigma$ in \eqref{cone}.}
\end{equation} 
Because the cone condition \eqref{cone} implies that $\sigma(0)=0$, 
the positivity principle for SPDEs implies that
\[
	\P\left\{ u(t\,,x)>0\text { for every $t\ge0$ and $x\in\mathbb{T}$}\right\}=1;
\]
see \cite{Mueller1}.

One of the interesting properties of \eqref{SHE} is that its solution
is \emph{intermittent} in the sense of \cite{BC, FK}. More precisely, intermittency (or moment intermittency) can be defined  as the property that 
\begin{equation}\label{g:gg:g}
	k\mapsto\frac{\overline\gamma(k)}{k}
	\quad\text{and/or}\quad
	k\mapsto\frac{\underline\gamma(k)}{k}\quad\text{is strictly increasing on 
	$[2\,,\infty)$},
\end{equation}
where $\underline\gamma,
\overline\gamma:[2\,,\infty)\to[-\infty\,,\infty]$ are given by 
\begin{equation*}
	\underline\gamma(k) :=\liminf_{t\to\infty} \frac1t 
	\inf_{x\in\mathbb{T}}\log\E\left( |u(t\,,x)|^k\right)
	\quad\text{and}\quad
	\overline\gamma(k) := \limsup_{t\to\infty} \frac1t 
	\sup_{x\in\mathbb{T}}\log\E\left( |u(t\,,x)|^k\right).
\end{equation*}
Here $\underline\gamma$ and $\overline \gamma$ are called \emph{lower and upper moment Lyapunov exponents} respectively. 

As a result of Jensen's inequality, it is easy to see that both
$k\mapsto \underline\gamma(k)/k$ and 
$k\mapsto \overline\gamma(k)/k$ are monotonically nondecreasing. So the
defining feature of intermittency is the \emph{strictness}
of this monotonicity. Indeed, Jensen's inequality for moments is strict iff 
the random variable is not constant over the probability space. 
In the setting of this paper,  intermittency is implied by the following, more easy-to-check, 
\emph{weak intermittency}
condition:
\begin{equation}\label{WI}
	0<\underline{\gamma}(k) \le \overline{\gamma}(k)<\infty
	\quad\text{for all $k\ge2$};
\end{equation}
see \cite{FK} for the relation between \eqref{g:gg:g} and \eqref{WI} and  also see \cite{FN, KK1, KK2, Nualart, Xie} for the moments and  weak intermittency of the solution $u$ to \eqref{SHE} on bounded intervals with various boundary conditions.   We have set things up so that \eqref{WI} is in fact
equivalent to the strict monotonicity of both
$k\mapsto \underline{\gamma}(k)/k$
and
$k\mapsto \overline{\gamma}(k)/k$.

In order to see intuitively how the moments give information about the peaks
of the solution, suppose $\underline\gamma=\overline\gamma$, and 
call their common value $\gamma$. This means roughly that, for every $k\ge2$,
\begin{equation}\label{interm:wish}
	\E\left( |u(t\,,x)|^k\right)\approx \e^{t\gamma(k)}\qquad\text{when
	$t\gg1$}.
\end{equation}
Because $k\mapsto \gamma(k)/k$ is strictly increasing on $[2\,,\infty)$,
there exist constants $2\le k_1<k_2<\cdots$, all strictly increasing, 
and events $A_1(t),A_2(t),\ldots$ (one for every $t>0$), and constants $C_1,C_2,\ldots>0$
such that:
\begin{enumerate}
	\item[({\bf I.1})] $\P(A_n(t))\le \exp(-C_nt)$ for all $n\ge1$ and $t\gg1$; and
	\item[({\bf I.2})] For all $n\ge1$ and $t\gg1$,	
		$\E\left(|u(t\,,x)|^{k_n}\right) \approx
		\E\left(|u(t\,,x)|^{k_n}\,;\, A_n(t)\right).$
\end{enumerate}
Indeed, by \eqref{g:gg:g} we can find for every $n\ge 1$ real numbers $a_n$
such that
\begin{equation}\label{interm:Molch}
	\frac{\gamma(k_{n-1})}{k_{n-1}} <a_n<
	\frac{\gamma(k_n)}{k_n},
\end{equation}
then set
\[
	A_n(t) := \left\{\omega\in\Omega:\
	\e^{a_n t} \le |u(t\,,x)(\omega)| \right\},
\]
and finally apply Chebyshev's inequality to deduce ({\bf I.1}):
\begin{align*}
	\P\left( A_n(t)\right) &\le \exp\left( -a_{n}k_{n-1}t\right)
		\E\left(|u(t\,,x)|^{k_{n-1}}\right)\\
	&\lessapprox\exp\left( -a_{n}k_{n-1}t + \gamma(k_{n-1})t\right)
		\qquad[\text{see }\eqref{interm:wish}]\\
	&\lessapprox\exp(-C_nt)\ \text{for some $C_n>0$}\
		[\text{see \eqref{interm:Molch}}].
\end{align*}
We deduce ({\bf I.2}) by noticing that
\begin{align*}
	\E\left( |u(t\,,x)|^{k_n}\,;\, \left[
		A_n(t)\right]^c\right) &\le \exp\left( k_na_nt\right)\\
	&\ll \exp\left( \gamma(k_n)t\right)\qquad[\text{see \eqref{interm:Molch}}]\\
	&\approx \E\left(|u(t\,,x)|^{k_n}\right)\quad[\text{see
		\eqref{interm:wish}}].
\end{align*}
From this simple heuristic about the Lyapunov exponents, we learn a good deal about the high peaks
of $u$, namely, that:
\begin{enumerate}
	\item The moments of the solution grow exponentially rapidly  as $t\to\infty$,
		and nearly all of the contribution to the $k_n$-th moment of $u(t\,,x)$
		comes from a small part [$A_n(t)$] of the probability space where
		$u(t\,,x)$ is unduly large; and
	\item The $k_1$-th, $k_2$-th, \dots 
		moments of $u(t\,,x)$ are influenced by decreasing 
		small parts of the underlying probability space.
\end{enumerate}
In other words, the high peaks tend to appear at large times, and
they tend to be highly localized in the probability space. This picture describes one part of ``physical intermittency'' in probability space where physical intermittency usually refers to the property that the solution $u$ tends to develop ``\emph{tall peaks},'' 
``\emph{distributed over
small islands},'' and ``\emph{separated by large areas where $u$ is small} (voids)'' (see \cite{BC,CM,CMS2002,GD,GT,ZMRS1988,ZMRS1985,ZMS,ZTPS}). 

The main goal of the present paper is to study the part of physical intermittency
that does \emph{not} seem to be a natural consequence of conditions
such as \eqref{g:gg:g} or \eqref{WI}. Namely, we currently propose to analyze the
``\emph{voids}'' (the event where $u$ is small). One of the key steps toward this goal is the following
result, which is the counterpart to \eqref{WI}.

\begin{theorem}\label{th:interm1}
There exist $t_0\geq 1$, an event $B(t)$ for every $t\geq t_0$, and constant $c>0$ which is independent of $t$ such that for every $k\geq 2$, there exist $c_{1,k},c_{2,k}>0$ such that: 
\begin{enumerate}
	\item $\P(B(t))\ge 1 - c\,\exp(-c t)$ for all $t\ge t_0$; and
	\item For all $t\ge t_0$,
		\[
			c_{2,k}\e^{-c_{1,k}t}
			\le \E\left( \inf_{x\in\mathbb{T}}
			    |u(t\,,x)|^k\,;\, B(t)\right)
			\le \E\left( \sup_{x\in\mathbb{T}}
			    |u(t\,,x)|^k\,;\, B(t)\right)
			\le c_{1,k}\e^{-c_{2,k}t}.
		\]
\end{enumerate}
\end{theorem} (The proof of this is in Section 8.)

In other words, we see from Theorem \ref{th:interm1} the following
property which contrasts with the earlier discussion about
moment intermittency and its consequences: 
For large values of $t$, only a tiny part of the probability space contributes to the 
moments of $u(t\,,x)$.  
In some sense, this property and moment intermittency
give us a complete mathematical description of the 	``physical intermittency''
of the solution $u$ in probability space.

In this connection, let us also mention a more precise
result. The following 
is a non-trivial pathwise variation of Theorem \ref{th:interm1},
which gives precise bounds on the a.s.\ dissipation of the solution to
\eqref{SHE}, viewed as the solution to a semi-linear heat-flow problem in 
the random environment $\xi$.
	
\begin{theorem}\label{th:decay:t}
	With probability one, 
	\[
		-\infty<\liminf_{t\to\infty} \frac1t \log \inf_{x\in\mathbb{T}}u(t\,,x)
		\le \limsup_{t\to\infty}\frac1t \log\sup_{x\in\mathbb{T}} u(t\,,x)<0.
	\]
	In particular, the positive random variable
	$\sup_{x\in\mathbb{T}}u(t\,,x)$ converges 
	a.s.\ to zero [fast] as $t\to\infty$.
\end{theorem}

Our analysis of Theorem \ref{th:interm1} and Theorem \ref{th:decay:t} hinges on a novel $L^1/L^\infty$
interpolation inequality, see Proposition \ref{pr:improved},
which is interesting in its own right.  Roughly speaking, we can 
control the supremum of our solution by its $L^1$ norm, and we 
can show using martingale arguments that, with high probability, 
the $L^1$ norm declines exponentially.  

Our analysis has other consequences too. 
For example, we can describe
the system \eqref{SHE} in the ``high-noise''
setting. That is, consider the SPDE \eqref{SHE} where we replace
$\sigma(u)$ by $\lambda\sigma(u)$ for a large constant $\lambda>0$,
as follows:
\begin{equation}\label{SHE-lambda}
	\partial_t u(t\,,x\,;\lambda) = \tfrac12 \partial^2_x u
	(t\,,x\,;\lambda) + \lambda\sigma(u(t\,,x\,;\lambda))\xi(t\,,x),
\end{equation}
with periodic boundary conditions on $[-1\,,1]$ and initial value $u_0$,
as before.
In other words, we simply replace the function $\sigma$ by
$\lambda\sigma$, and add $\lambda$ to the notation for $u$
to help keep better track of this change. Since $\lambda\sigma$ is also Lipschitz continuous
and satisfies \eqref{cone}, all of this is merely recording a change in the notation.

Now we can state a
result about the large-noise behavior of
the solution to \eqref{SHE}, equivalently the large-$\lambda$ behavior
of the solution to \eqref{SHE-lambda}. Roughly speaking, the following theorem states
that if the level $\lambda$ of the noise is high then voids take over
rapidly, with very high probability. More precisely, we have

\begin{theorem}[Large-noise regime]\label{th:lambda:UB}
	For every $t>0$,
	\[
		\limsup_{\lambda\uparrow\infty}\frac{1}{\lambda^2}
		\log \P\left\{ \sup_{x\in\mathbb{T}}u(t\,,x\,;\lambda) > 
		\exp\left(-\frac{{\rm L}_\sigma^2\lambda^2 t}{64}
		\right)\right\}\le-\frac{{\rm L}_\sigma^2t}{64}.
	\]
	In particular, for every $t>0$ fixed,  the positive random variable
	$\sup_{x\in\mathbb{T}}u(t\,,x\,;\lambda)$ converges 
	in probability to zero [fast] as $\lambda\to\infty$.
\end{theorem}

We conclude the Introduction by setting forth some notation that will be
used throughout the paper.

In order to simplify some of the formulas, we distinguish between the spaces
$L^k(\mathbb{T})$ and $L^k(\P)$ by writing the former as 
\[
	L^k := L^k(\mathbb{T})\qquad [1\le k<\infty].
\]
Thus, for example, if $f\in L^k$ for some $1\le k<\infty$, then
$\|f\|_{L^k} = [ \int_{-1}^1 |f(x)|^k\,\d x]^{1/k}.$
We will abuse notation slightly and write
$\|f\|_{L^\infty}:= \sup_{x\in\mathbb{T}}|f(x)|,$
in place of the more customary
essential supremum. 
The $L^k(\P)$-norm
of a random variable $Z\in L^k(\P)$ is denoted by
$\|Z\|_k := \{ \E\left(|Z|^k\right)\}^{1/k}$
for all $1\le k<\infty$.


\section{The mild solution}\label{sec:mild}
Consider the SPDEs \eqref{SHE} and \eqref{SHE-lambda}. Because
$u(t\,,x)=u(t\,,x\,;1)$, it suffices to consider only the SPDE
\eqref{SHE-lambda} for a general $\lambda>0$. We shall do so tacitly
from here on.

Let $W=\{W(t\,,x)\}_{t\ge0,x\in\mathbb{T}}$ denote a two-parameter
Brownian sheet; that is, $W$ is a two-parameter, centered, generalized Gaussian
random field with 
\[
	\Cov\left[ W(t\,,x) \,,\,W(s\,,y)\right] = \min(s\,,t)\min(x\,,y)
	\qquad\text{for all $s,t\ge0$ and $x,y\in\mathbb{T}$}.
\]
It is well known (see \cite[Theorem 1.1]{Walsh}) that $W$ has continuous 
trajectories (up to a modification).  Therefore,
\[
	\xi(t\,,x) = \partial_t\partial_x W(t\,,x)
\]
exists as a generalized random function. This $\xi$ is space-time
white noise, and was mentioned already in the Introduction.

Let $(\tau\,;x\,,y)\mapsto p_{\tau}(x\,,y)$ denote the fundamental
solution to the heat operator $\partial_t -\partial_x^2$ on $(0\,,\infty)\times\mathbb{T}$
with periodic boundary conditions and initial data $p_0(x,y)=\delta(x-y)$, 
where $\delta$ is the Dirac delta function. That is,
\begin{equation}\label{eq:p}
	p_\tau(x\,,y) := \sum_{n=-\infty}^\infty G_\tau(x-y+2n)\qquad
	[\tau>0,\, x,y\in\mathbb{T}],
\end{equation}
where $G$ is the heat kernel in free space; that is,
\begin{equation}\label{eq:G}
	G_\tau(a) := (4\pi\tau)^{-1/2}\exp\left( - \frac{a^2}{4\tau}\right)
	\qquad\text{for all $\tau>0$ and $a\in\R$.}
\end{equation}
Also, let $\{P_t\}_{t\ge0}$ denote the corresponding heat 
semigroup. That is, $P_0f:=f$ for every measurable and bounded function 
$f:\mathbb{T}\to\R_+$,  and
\begin{equation}\label{P_tf}
	(P_tf)(x) := \int_{-1}^1 p_t(x\,,y)f(y)\,\d y,
\end{equation}
for all $t>0$ and $x\in\mathbb{T}$.

With the preceding notation in place,
we then follow Walsh \cite[Chapter 3]{Walsh}
and interpret \eqref{SHE-lambda} in mild/integral form as follows:
\begin{equation}\label{mild}
	u(t\,,x\,;\lambda) = (P_tu_0)(x) +  \cI_t(x\,;\lambda),
\end{equation}
where $\mathcal{I}$ is defined pointwise as the Walsh stochastic integral,
\begin{equation}\label{I}
	\cI_t(x\,;\lambda) := \lambda\int_{(0,t)\times\mathbb{T}} 
	p_{t-s}(x\,,y)\sigma(u(s\,,y\,;\lambda))\, W(\d s\,\d y),
\end{equation}
for every $t,\lambda>0$ and $x\in\mathbb{T}$.
As was mentioned in the Introduction (for $\lambda=1$),
it is well known \cite[Th.\ 3.2, p.\  313, and Cor.\ 3.4, p.\ 318]{Walsh} that there exists a unique
weak solution of \eqref{SHE-lambda} that is continuous and satisfies \eqref{mild},
as well as the following moment condition:
\begin{equation}\label{walsh:moments}
	\sup_{\substack{x\in\mathbb{T}\\t\in(0,T]}}
	\E\left(|u(t\,,x\,;\lambda)|^k\right)<\infty
	\qquad[0<T<\infty,\ 1\le k<\infty].
\end{equation}
Moreover, for every $\lambda>0$,
\begin{equation}\label{positive}
	\P\left\{ u(t\,,x\,;\lambda)> 0\ \text{for all $t\ge0$ and $x\in\mathbb{T}$}\right\}=1.
\end{equation}
In the case that $\sigma(z):=\text{const}\cdot z$ for all $z\in\R$, this follows
from Theorem 1 of Mueller \cite{Mueller1}. The general case follows by making
modifications to the proof of that theorem; see the proof of Theorem 1.7 of
Conus et al \cite{CJK}.


\section{The total mass process}

We may integrate both sides of \eqref{mild} $[\d x]$ in order to see that
\begin{equation}\label{M:MG}
	\|u(t\,,\cdot\,;\lambda)\|_{L^1}= \|u_0\|_{L^1} + \lambda
	\int_{(0,t]\times\mathbb{T}} \sigma(u(s\,,y\,;\lambda))\, W(\d s\,\d  y)
	\qquad[t\ge0].
\end{equation}
The interchange of the integrals is justified by an appeal to a stochastic
Fubini theorem \cite[Th.\ 2.6, p.\ 296]{Walsh}. Thus, it follows from \eqref{cone}, 
\eqref{walsh:moments}, \eqref{positive} and \eqref{M:MG} that 
$t\mapsto\|u(t\,,\cdot\,;\lambda)\|_{L^1}$ defines a positive, continuous, $L^2$-martingale.
The following result
ensures that the said martingale decays  exponentially
rapidly at rate not less than a fixed multiple of $\lambda^2$.

\begin{proposition}\label{pr:M:small}
	For every $t,\lambda>0$ and $\varepsilon\in(0\,,1)$,
	\[
		\P\left\{\|u(s\,,\cdot\,;\lambda)\|_{L^1}
		\ge \|u_0\|_{L^1}\exp\left(-\frac{(1-\varepsilon)
		\lambda^2{\rm L}_\sigma^2 s}{4}\right)
		\text{ for some $s\ge t$}\right\} \le 
		\exp\left( -\frac{\varepsilon^2\lambda^2{\rm L}_\sigma^2t}{16}\right).
	\]
\end{proposition}

The proof of Proposition \ref{pr:M:small}
requires a basic lemma about continuous martingales,
which might be of independent interest.

\begin{lemma}\label{lem:MG}
	Let $X=\{X_t\}_{t\ge0}$ be a continuous $L^2(\P)$ martingale, and suppose 
	there is a nonrandom $c>0$
	such that $\<X\>_t\ge ct$ for all $t\ge0$, a.s.
	Then, for all nonrandom constants $\varepsilon,T>0$,
	\[
		\P\left\{ X_t \ge \varepsilon\< X\>_t\text{ for some $t\ge T$}\right\}
		\le \exp\left(-\frac{cT\varepsilon^2}{2}\right).
	\]
\end{lemma}

\begin{proof}
        Recall that a continuous local martingale such as $X$ is a 
        time-change of a Brownian motion $\{B(s)\}_{s\ge0}$
	(see \cite[Theorem 1.6, page 181]{RY}), so that $X_t = B(\<X\>_t)$ for all $t\ge0$, a.s.
	We first note that
	\[
		\P\left\{ X_t \ge \varepsilon\< X\>_t\text{ for some $t\ge T$}\right\}
		\le \P\left\{ \sup_{s\ge cT}\frac{B(s)}{s}\ge\varepsilon\right\}.
	\]
	Next we note that $\{B(s)/s\}_{s>0}$ has the same
	law as $\{B(1/s)\}_{s>0}$ thanks to Brownian time inversion. Thus
	\[
		\P\left\{ X_t \ge \varepsilon\< X\>_t\text{ for some $t\ge T$}\right\}
		\le \P\left\{ \sup_{r\le 1/(cT)}B(r)\ge\varepsilon\right\}.
	\]
        Because
	$(2\pi)^{-1/2}\int_a^\infty\exp(-x^2/2)\,\d x\le (1/2)\exp(-a^2/2)$ for all $a>0$,
	the reflection principle implies the result.
\end{proof}

Armed with Lemma \ref{lem:MG},
we conclude the section with the following.

\begin{proof}[Proof of Proposition \ref{pr:M:small}]
	In the case that $\sigma(z)\equiv \text{const}\cdot z$ for all $z\in\R$
	and the SPDE \eqref{SHE} has Dirichlet---instead of periodic---boundary conditions,
	Mueller and Nualart \cite[Theorem 2]{MN} have proved that
	$\E(|u(t\,,x\,;\lambda)|^{-k})<\infty$ for all $1\le k<\infty$, $t>0$, and $x\in\mathbb{T}$.
	Their argument, in fact, proves that, in the present setting,\footnote{
			There is an extension of the method of Mueller and Mueller--Nualart
			\cite{Mueller1,MN}---see 
			the proof of Theorem 1.7 of Conus et al \cite{CJK}---that
			proves \eqref{eq:neg:moments} in the present more general 
			choices of $\sigma$ for SPDEs on $\R_+\times\R$. 
			The latter argument works in exactly the same way
			in the present setting.
		} 
	\begin{equation}\label{eq:neg:moments}
		\E\left(\left| \inf_{x\in\mathbb{T}} u(t\,,x\,;\lambda)\right|^{-k}\right)<\infty
		\qquad[t\ge0,\ 1\le k<\infty].
	\end{equation}
	
	Let us define 
	\[
		M_t : =\|u(t\,,\cdot\,;\lambda)\|_{L^1} = \int_{-1}^1
		u(t\,,x\,;\lambda)\,\d x\qquad[t\ge0],
	\]
	and infer from \eqref{eq:neg:moments} that
	\begin{equation}\label{neg:moments}
		\E\left( M_t^{-k}\right)<\infty
		\qquad\text{for all $t\ge0$ and $1\le k<\infty$,}
	\end{equation}
	We will use \eqref{neg:moments} several times, sometimes tacitly, in the sequel.
	
	We can apply It\^o's formula in order to see that, a.s.,
	\[
		\log  M_t = \log  M_0 + \int_0^t  M_s^{-1}\,\d  M_s - \frac12\int_0^t  M_s^{-2}\,
		\d\< M\>_s\qquad\text{for all $t\ge0$}.
	\]
	Define
	\[
		 N_t := \int_0^t  M_s^{-1}\,\d  M_s\qquad
		 \text{for all $t\ge0$}.
	\]
        Let $\{\cF_t\}_{t\ge0}$ denote the filtration generated by $W(s\,,\cdot)$ for $s\leq t$.
	Then, clearly, $ N:=\{N_t\,,\cF_t\}_{t\ge0}$ is a continuous $L^2$-martingale 
	with quadratic variation
	$\< N\>_t = \int_0^t  M_s^{-2}\,\d\< M\>_s$ at time $t>0$.
	In other words, $\log  M_t = \log  M_0 +  N_t - \frac12\< N\>_t$ a.s.\
	for all $t>0$;
	this is another way to say that
	\begin{equation}\label{M:exp}
		 M_t =  M_0\,\exp\left( N_t-\tfrac12\< N\>_t\right)
		 \qquad\text{a.s.\ for all $t>0$}.
	\end{equation}
	That is, $ M$ is the exponential martingale of the martingale $N$,
	and $M$ is initialized at $ M_0$.
	
	We examine the quadratic variation of $ N$ more closely next:
	\begin{equation}\label{<X>:LB}\begin{split}
		\frac{\< N\>_t}{t} &= \frac{\lambda^2}{t}\int_0^t
			\frac{\d s}{ M_s^2}\int_{-1}^1\d y\
			[\sigma(u(s\,,y\,;\lambda))]^2\\
		&\ge \frac{\lambda^2}{t}\int_0^t
			\frac{\d s}{ M_s^2}\int_{-1}^1\d y\
			{\rm L}_\sigma^2|u(s\,,y\,;\lambda)|^2\\
		&\ge \frac{\lambda^2{\rm L}_\sigma^2}{2},
	\end{split}\end{equation}
	owing to Condition \eqref{cone} and the Cauchy--Schwarz inequality. 
	In light of \eqref{M:exp},
	\[
		\P\left\{\exists s\ge t:\  M_s\ge M_0\e^{-\beta s}\right\}
		= \P\left\{\exists s\ge t:\  N_s\ge\tfrac12\< N\>_s-\beta s\right\},
	\]
	for all $\beta,t>0$.
	Therefore, we may first use \eqref{<X>:LB} and then appeal to
	Lemma \ref{lem:MG} in order to see  that,
	as long as $0<\beta<\lambda^2{\rm L}_\sigma^2/4,$
	\begin{align*}
		\P\left\{\exists s\ge t:\  M_s\ge M_0\e^{-\beta s}\right\}
			&\le \P\left\{ N_s\ge\left(\frac12-\frac{2\beta}{\lambda^2{\rm L}_\sigma^2}
			\right)\< N\>_s\text{ for
			some $s\ge t$}\right\}\\
		&\le \exp\left( -\frac{\lambda^2{\rm L}_\sigma^2t}{4}
			\left(\frac12-\frac{2\beta}{\lambda^2{\rm L}_\sigma^2}
			\right)^2\right)  \\
		&=\exp\left( -t\lambda^{-2}{\rm L}_\sigma^{-2}
			\left(\frac{\lambda^2{\rm L}_\sigma^2}{4}-
			\beta\right)^2\right).
	\end{align*}
	Substitute $\beta=\tfrac14 (1-\varepsilon)\lambda^2{\rm L}_\sigma^2$
	to deduce Proposition \ref{pr:M:small}.
\end{proof}


\section{Regularity}

In order to prove the announced regularity properties of the solution
$u$ to \eqref{SHE-lambda} we first require a moment bound, with explicit
constants, for the solution $u$.

\begin{proposition}\label{pr:moment:UB}
	Choose and fix a real number $c>48$. Then,
	for all real numbers $k\ge 2$ and $\lambda>0$
	that satisfy $k\lambda^2 \ge (c\lip^2_\sigma)^{-1},$
	the following holds: Uniformly for all $t>0$,
	\[
		\sup_{x\in\mathbb{T}} \E\left( |u(t\,,x\,;\lambda)|^k\right) \le 2^{k/2}
		\left(1 -\frac{48}{c}\right)^{-k/2}\|u_0\|_{L^\infty}^k
		\cdot \exp\left(\frac{c^2}{2}\lip_\sigma^2 k^3\lambda^4t\right).
	\]
\end{proposition}

Proposition \ref{pr:moment:UB} implies also \eqref{walsh:moments}.

\begin{proof}
	We modify some of the ideas of Foondun
	and Khoshnevisan \cite{FK}, but need to make a series of
	modifications. Define
	\begin{equation}\label{vartheta:c}
		\vartheta := c^2\lip^4_\sigma k^2\lambda^4,
	\end{equation}
	where, $c>48$ is large enough to ensure that $\vartheta\ge 1
$ whenever $k\lambda^2\geq(c{\rm Lip}_\sigma^2)^{-1}$ holds.  
	
	For all $t\ge0$ and $-1\le x\le 1$,
	let $u^{(0)}_t(x) := u_0(x\,;\lambda)$ and define iteratively for all $n\ge0$,
	\begin{equation}\label{mild:1}
		u^{(n+1)}_t(x\,;\lambda) = (P_tu_0)(x) + \cI_t^{(n)}(x),
	\end{equation}
	where $\{P_t\}_{t\ge0}$ continues to denote the heat semigroup---see
	\eqref{P_tf}---and 
	\[
		\cI^{(n)}_t(x)=\cI^{(n)}_t(x\,;\lambda) 
                := \lambda\int_{(0,t)\times\mathbb{T}} 
		p_{t-s}(x\,,y)\sigma(u^{(n)}_s(y))\, W(\d s\,\d y).
	\]
	The random field $(t\,,x)\mapsto
	u^{(n)}_t(x)$ is the $n$th-stage Picard-iteration approximation
	of $u(t\,,x\,;\lambda)$. 
	
	It is well known (see \cite[Ch.\ 3]{Walsh}) that
	\begin{equation}\label{Picard}
		\lim_{n\to\infty}\sup_{t\in(0,T]}\sup_{x\in\mathbb{T}}
		\E\left( \left| u^{(n)}_t(x) - u(t\,,x) \right|^k \right)=0,
	\end{equation}
	and
	\[
		\lim_{n\to\infty}\sup_{t\in(0,T]}\sup_{x\in\mathbb{T}}
		\E\left( \left| \cI^{(n)}_t(x) - \cI_t(x\,;\lambda) \right|^k \right)=0,
	\]
	for all $T\in(0\,,\infty)$ and $k\in[1\,,\infty)$.
	
	Since the semigroup $\{P_t\}_{t\ge0}$ is conservative,
	$(P_tu_0)(x)\le\|u_0\|_{L^\infty}$ for all $t\ge0$
	and $x\in\mathbb{T}$. Therefore, \eqref{mild:1} 
	implies that for all integers $n\ge 0$
	and real numbers  $k\in[2\,,\infty)$, $t>0$, and $x\in\mathbb{T}$,
	\begin{equation}\label{u:u:I}
		\left\|u^{(n+1)}_t(x)\right\|_k \le  
		\|(P_tu_0)(x)\|_k + \left\|\cI^{(n)}_t(x)\right\|_k\le  
		\|u_0\|_{L^\infty} + \left\|\cI^{(n)}_t(x)\right\|_k.
	\end{equation}
	A Burkholder-Davis-Gundy-type inequality
	for stochastic convolutions (see \cite[Pr.\ 4.4, p.\ 36]{cbms})
	then yields the following inequality:
	\begin{equation}\label{BOPO}\begin{split}
		\left\|\cI^{(n)}_t(x)\right\|_k &\le 
			\sqrt{4k\lambda^2
			\int_0^t\d s\int_{-1}^1\d y\ \left[ p_{t-s}(x\,,y)\right]^2
			\left\| \sigma(u^{(n)}_s(y))\right\|_k^2}\\
		&\le \lip_\sigma
			\sqrt{4k\lambda^2\int_0^t\d s\int_{-1}^1\d y\ \left[ p_{t-s}(x\,,y)\right]^2
			\left\|u^{(n)}_s(y)\right\|_k^2}.
	\end{split}\end{equation}
	By the Chapman--Kolmogorov equation and symmetry,
	\begin{equation}\label{ppp}\begin{split}
		\int_{-1}^1 \left[p_{t-s}(x\,,y) \right]^2\,\d y &= 
			\int_{-1}^1 p_{t-s}(x\,,y) p_{t-s}(y\,,x)\,\d y
			= p_{2(t-s)}(x\,,x)\\
		&\le 2\left( \frac{1}{\sqrt{t-s}}+1\right);
	\end{split}\end{equation}
	the final estimate is justified by Lemma \ref{lem:c_l} below.
	Let us define 
	\[
		\psi^{(n)}(t) := \sup_{x\in\mathbb{T}}\left\|u_t^{(n)}(x)\right\|_k^2
		\qquad\text{for all $t\ge0$ and integers $n\ge 0$}.
	\]
	We can combine \eqref{BOPO} and \eqref{ppp} and use the elementary 
        inequality $(a+b)^2\leq 2a^2+2b^2$ to see that
	\[
		\psi^{(n+1)}(t) \le 2\|u_0\|_{L^\infty}^2
		+ 16\lip_\sigma^2
		k\lambda^2\int_0^t \psi^{(n)}(s)\left(\frac{1}{\sqrt{t-s}}+1\right)\d s.
	\]
	Multiply both sides by $\exp(-\vartheta t)$ in order
	to see that
	\[
		\Psi_n  := \sup_{t\ge0} \left[\e^{-\vartheta t}\psi^{(n)}(t)\right]
	\]
	satisfies
	\begin{align*}
		\Psi_{n+1}  &\le 2\|u_0\|_{L^\infty}^2 + 16\lip_\sigma^2\Psi_n 
			k\lambda^2\sup_{t\ge0}\int_0^t \e^{-\vartheta(t-s)}\left(\frac{1}{\sqrt{t-s}}+1\right)\d s\\
		&\le2\|u_0\|_{L^\infty}^2 + 16\lip_\sigma^2\Psi_n 
			k\lambda^2\int_0^\infty \e^{-\vartheta r}\left(\frac{1}{\sqrt{r}}+1\right)\d r\\
		&= 2\|u_0\|_{L^\infty}^2 + 16\lip_\sigma^2\Psi_n 
			k\lambda^2\left( \sqrt{\frac{\pi}{\vartheta}} + \frac{1}{\vartheta}\right).
	\end{align*}
	Because $\vartheta\ge 1$, we have
	$\sqrt{\pi/\vartheta}+\vartheta^{-1}\le 3/\sqrt\vartheta$, and hence
	\begin{align*}
		\Psi_{n+1} 
			\le 2\|u_0\|_{L^\infty}^2 + \frac{48\lip_\sigma^2 
			k\lambda^2}{\sqrt{\vartheta}}\Psi_n &= 2\|u_0\|_{L^\infty}^2
			+\frac{48}{c}\Psi_n
			\quad\text{for all $\alpha\ge 1$ and
			$n\ge 0$}.
	\end{align*}
	The second line follows from the first, thanks to \eqref{vartheta:c} and the fact
	that $c>48$.
	Because $\Psi_0=\sup_{x\in\mathbb{T}}u_0(x)$ is finite, the preceding implies
	that $\sup_{n\ge0}\Psi_n<\infty$, and
	\[
		\limsup_{n\to\infty}\Psi_n \le 2\|u_0\|_{L^\infty}^2
		\left(1 -\frac{48}{c}\right)^{-1}.
	\]
	According to \eqref{Picard} and Fatou's lemma,
	\[
		\limsup_{n\to\infty}\Psi_n \ge \sup_{t\ge0}\sup_{x\in\mathbb{T}}\left[
		\e^{-\vartheta t}\|u(t\,,x\,;\lambda)\|_k^2 \right].
	\]
	Therefore, we may combine the preceding two displays, all the time
	remembering our choice of $\vartheta$, in order to conclude that
	\[
		\|u(t\,,x\,;\lambda)\|_k^2\le 2\|u_0\|_{L^\infty}^2\e^{\vartheta t}
		\left(1 -\frac{48}{c}\right)^{-1},
	\]
	uniformly for all $-1\le x\le 1$ and $t>0$, and all $k\ge 2$ and
	$\lambda>0$ that ensure that $\vartheta\ge 1$.
	This is another way to state the proposition.
\end{proof}

We now use our moment bound [Proposition \ref{pr:moment:UB}] to establish
the regularity of $\lambda\mapsto u(t\,,x\,;\lambda)$.

\begin{proposition}\label{pr:modulus:UB}
	Choose and fix a real number $c>48$. Then,
	for all real numbers $k\ge 2$ and $\alpha,\beta>0$
	that satisfy $k(\alpha\vee\beta)^2 \ge (c\lip^2_\sigma)^{-1},$
	the following holds: Uniformly for all $t>0$,
	\[
		\sup_{x\in\mathbb{T}}
		\E\left(|u(t\,,x\,;\alpha) - u(t\,,x\,;\beta) |^k\right)
		\le L_c^{k/2} \|u_0\|_{L^\infty}^k
		\exp\left( \frac{c^2}{2}
		\lip^4_\sigma k^3(\alpha\vee\beta)^4t\right)
		\cdot 
		\frac{|\alpha-\beta|^k}{(\alpha\wedge\beta)^k},
	\]
	with $L_c:=(96/c)(1-(48/c))^{-2}$.
\end{proposition}

\begin{remark}\label{rm:cont}
	Standard methods---see \cite[Chapter 3]{Walsh}---show that
	$(t\,,x)\mapsto u(t\,,x\,;\lambda)$ has a continuous modification
	for every $\lambda>0$. In fact,
	for every $\varepsilon\in(0\,,1)$, $k\ge 2$,
	$T>t_0>0$ and $\Lambda>0$,
	\[
		\sup_{\lambda\in(0\,,\Lambda)}
		\left\|\sup_{\substack{-1\le x\neq y\le 1\\t_0<s\neq t<T}}
		\frac{|u(t\,,x\,;\lambda)-u(s\,,y\,;\lambda)|}{|x-y|^{(1-\varepsilon)/2}
		+|s-t|^{(1-\varepsilon)/4}}\right\|_k<\infty.
	\]
	One has to be somewhat careful here since, unlike the standard
	theory \cite{Walsh}, we may not choose $t_0$ to be zero here.
	The details can be found in Proposition \ref{pr:simultaneous} below.
	In any case, we can see from Proposition \ref{pr:modulus:UB} and an appeal to the
	Kolmogorov continuity theorem [i.e., a chaining argument]
	that: (i) $(t\,,x\,,\lambda)\mapsto u(t\,,x\,;\lambda)$ 
	has a H\"older-continuous modification on 
	$\R_+\times\mathbb{T}\times(0\,,\infty)$;
	and (ii) That modification satisfies the following for every $p,q,r\in(0\,,1)$,
	$k\ge 2$, $\Lambda>\lambda>0$, and $T>t_0>0$:
	\[
		\left\| \sup_{\substack{-1\le x\neq y\le 1\\
		t_0\le s\neq t\le T\\\lambda\le\alpha\neq\beta\le \Lambda}}
		\frac{|u(t\,,x\,;\alpha)-u(s\,,y\,;\beta)|}{|x-y|^{p/2}
		+|s-t|^{q/4}+|\alpha-\beta|^r}\right\|_k<\infty.
	\]
\end{remark}

To paraphrase Walsh \cite{Walsh}, the process $\lambda\mapsto u(t\,,x\,;\lambda)$
comes tantalizingly close to being Lipschitz continuous. One can elaborate
on this further as follows: Define
\[
	\sD(t\,,x\,;\lambda) := \frac{\partial}{\partial\lambda} u(t\,,x\,;\lambda),
\]
where the $\lambda$-derivative is understood in the sense of distributions,
and exists because $u$ is a continuous function of $\lambda$
[up to a modification]; see the preceding remark. According to
Rademacher's theorem, because $\sigma$ is Lipschitz continuous,
it has a weak derivative $\sigma'\in L^\infty(\mathbb{T})$.
Then, one can appeal to a 
stochastic Fubini argument in order to see that $\sD$ is the unique solution to
the $\lambda$-a.e.-defined stochastic integral equation,
\begin{multline*}
	\sD(t\,,x\,;\lambda) 
		= \int_{(0,t)\times\mathbb{T}} p_{t-s}(x\,,y)
		\sigma(u(s\,,y\,;\lambda))\, W(\d s\,\d y) \\
	+\lambda\int_{(0,t)\times\mathbb{T}} p_{t-s}(x\,,y)\sigma'(u(s\,,y\,;\lambda))
		\sD(s\,,y\,;\lambda)\, W(\d s\,\d y).
\end{multline*}
It is not difficult to show that if $\sigma$ has additional regularity properties---for 
instance, if $\sigma'$ is Lipschitz continuous---then $\sD$ is almost surely
H\"older-continuous in its three variables [up to a modification]. 
This proves the following: 
\begin{prop}
	If $\sigma\in C^1(\R)$ has a Lipschitz-continuous derivative, then
	$\lambda\mapsto u(t\,,x\,;\lambda)$ is a.s.\ continuously differentiable
	for every $t\ge0$ and $x\in\mathbb{T}$.
\end{prop}
We do not know whether the Lipschitz-continuity of $\sigma$ is really needed
for this differentiability result.

\begin{proof}[Proof of Proposition \ref{pr:modulus:UB}]
	Without loss of generality, we assume throughout that
	$\alpha>\beta$.
	
	We can write
	\begin{equation*}
		u(t\,,x\,;\alpha)-u(t\,,x\,;\beta) = \cI_t(x\,;\alpha) - \cI_t(x\,;\beta)
		= \cT_1 + \cT_2,
	\end{equation*}
	where
	\begin{align*}
		\cT_1 &:= \alpha\int_{(0,t)\times\mathbb{T}} p_{t-s}(x\,,y)
			\left[\sigma(u(s\,,y\,;\alpha))-\sigma(u(s\,,y\,;\beta))\right]
			W(\d s\,\d y),\\
		\cT_2 &:= (\alpha-\beta)\int_{(0,t)\times\mathbb{T}} p_{t-s}(x\,,y)
			\sigma(u(s\,,y\,;\beta))\, W(\d s\,\d y).
	\end{align*}
	Although $\cT_1$ and $\cT_2$ both depend on $(x\,,t\,,\alpha,\beta)$,
	we have not written those parameter dependencies explicitly
	in order to ease the typography.
	
	Define
	\[
		\cD^2 := \sup_{s\ge 0}\sup_{y\in\mathbb{T}}\left[
		\e^{-\vartheta s}\left\| u(s\,,y\,;\alpha) - u(s\,,y\,;\beta) \right\|_k^2\right],
	\]
	where $\vartheta$ is defined as in \eqref{vartheta:c}, but
	with a small difference; namely,
	\[
		\vartheta := c^2\lip^4_\sigma k^2\alpha^4.
	\]
	Our condition on $c$ is that $c>48$ is large enough to ensure that
	$\vartheta\ge 1$.
	
	We apply the Burkholder-Davis-Gundy-type inequality, 
	\cite[Pr.\ 4.4, p.\ 36]{cbms}, in order to see that
	\begin{align*}
		\|\cT_1\|_k^2 &\le 4k\alpha^2\lip_\sigma^2\int_0^t\d s\int_{-1}^1\d y\
			[p_{t-s}(x\,,y)]^2 \left\| u(s\,,y\,;\alpha) - u(s\,,y\,;\beta) \right\|_k^2\\
		&\le 4k\alpha^2\lip_\sigma^2\cD^2\e^{\vartheta t}
			\int_0^t\e^{-\vartheta s}
			\,\d s\int_{-1}^1\d y\ [p_s(x\,,y)]^2\\
		&= 4k\alpha^2\lip_\sigma^2\cD^2 \e^{\vartheta t}\int_0^t
			p_{2s}(x\,,x)\e^{-\vartheta s}\,\d s,
	\end{align*}
	thanks to the Chapman--Kolmogorov equation and argument in Proposition \ref{pr:moment:UB}. Lemma \ref{lem:c_l} below ensures
	the following:
	\begin{align*}
		\|\cT_1\|_k^2 &\le 8k\alpha^2\lip_\sigma^2\cD^2
			\e^{\vartheta t} \int_0^\infty
			\left(\frac{1}{\sqrt{s}}+1\right)\e^{-\vartheta s}\,\d s\\
		&=  8k\alpha^2\lip_\sigma^2\cD^2\e^{\vartheta t}
			\left(\sqrt{\frac{\pi}{\vartheta}}
			+\frac{1}{\vartheta}\right)\\
		&\le 
			\frac{24\alpha^2 k\lip_\sigma^2\cD^2\e^{\vartheta t}}{\sqrt{\vartheta}}\\
		&=\frac{24}{c}\cD^2\e^{\vartheta t}.
	\end{align*}
	We proceed in like manner to estimate the moments of $\cT_2$. First, note
	that, because $\alpha>\beta$,
	a Burkholder-Davis-Gundy bound and Proposition \ref{pr:moment:UB} together imply that
	\begin{equation*}\label{T_2}\begin{split}
		\|\cT_2\|_k^2 &\le 4k(\alpha-\beta)^2\lip_\sigma^2\int_0^t\d s\int_{-1}^1\d y\
			[p_{t-s}(x\,,y)]^2\|u(s\,,y\,;\beta)\|_k^2\\
		&\le 8k \left(1-\frac{48}{c}\right)^{-1} (\alpha-\beta)^2\lip_\sigma^2\|u_0\|_{L^\infty}^2
			\int_0^t \d s\ e^{\vartheta s}\int_{-1}^1 \d y\ [p_{t-s}(x\,,y)]^2.
	\end{split}\end{equation*}
	Therefore, after making a change of variables, we appeal first to the Chapman--Kolmogorov and then to Lemma
	\ref{lem:c_l} below in order to deduce the following:
	\begin{align*}\label{T_2:bis}
		\|\cT_2\|_k^2 &\le 8\left(1-\frac{48}{c}\right)^{-1}k(\alpha-\beta)^2\lip_\sigma^2\|u_0\|_{L^\infty}^2
			\e^{\vartheta t}\int_0^t\e^{-\vartheta s}p_{2s}(x\,,x)\,\d s\\\notag
		&\le 16\left(1-\frac{48}{c}\right)^{-1}k(\alpha-\beta)^2\lip_\sigma^2\|u_0\|_{L^\infty}^2
			\e^{\vartheta t}\int_0^\infty\e^{-\vartheta s}
			\left(\frac{1}{\sqrt s}+1\right)\d s \\\notag
		&\le \frac{48}{c\alpha^2}\left(1-\frac{48}{c}\right)^{-1}(\alpha-\beta)^2
			\|u_0\|_{L^\infty}^2\e^{\vartheta t}.
	\end{align*}
	We can now collect terms to find that
	\begin{equation*}\label{u:u:D}\begin{split}
		\left\| u(t\,,x\,;\alpha)-u(t\,,x\,;\beta)\right\|_k^2
			&\le 2\|\cT_1\|_k^2 + 2\|\cT_2\|_k^2 \\
		&\le \frac{48}{c}\cD^2\e^{\vartheta t} + 
			\frac{96}{c\alpha^2}\left(1-\frac{48}{c}\right)^{-1}
			(\alpha-\beta)^2\|u_0\|_{L^\infty}^2\e^{\vartheta t}.
	\end{split}\end{equation*}
	This bound holds pointwise. Therefore, we can divide both sides by $\exp(\vartheta t)$
	and optimize both sides over $x\in\mathbb{T}$ in order to conclude that
	\[
		\cD^2 \le\frac{48}{c}\cD^2 + 
		\frac{96}{c\alpha^2}\left(1-\frac{48}{c}\right)^{-1}(\alpha-\beta)^2\|u_0\|_{L^\infty}^2.
	\]
	Since $\alpha>\beta$ and $\vartheta\ge 1$, we may appeal to Proposition
	\ref{pr:moment:UB}---with $\lambda$ there replaced by $\alpha$ here---in order
	to see that $\cD<\infty$. In particular, because $c>48$, we find that
	\[
		\cD^2 \le 
		\frac{96\|u_0\|_{L^\infty}^2}{c}
		\left(1-\frac{48}{c}\right)^{-2}
		\frac{(\alpha-\beta)^2}{\alpha^2}.
	\]
	This is another way to state the proposition.
\end{proof}


\section{Improved regularity via interpolation}

In this section we use interpolation arguments to improve 
the moments estimates of the preceding sections and introduce
new moment estimates that, among other things, justify also 
Remark \ref{rm:cont}. One of the consequences of the matter that follows is
this:

\begin{proposition}\label{pr:simultaneous}
	The process $(t\,,x\,,\lambda)\mapsto u(t\,,x\,;\lambda)$
	has a continuous modification, indexed also by 
	$(0\,,\infty)\times\mathbb{T}\times(0\,,\infty)$, that weakly solves \eqref{SHE-lambda}
	outside of a null set that does not depend on $(t\,,x\,,\lambda)$.
\end{proposition}

The following will be the main result of this section. 
\begin{proposition}\label{pr:improved}
	There exists $\varepsilon_0=\varepsilon_0(\lip_\sigma)\in(0\,,1)$, small enough, such that 
	for every $\varepsilon\in(0\,,\varepsilon_0)$ and $t_0\geq 1$ there
	exist finite constants $C_1=C_1(\varepsilon\,,\lip_\sigma)>0$
	and $C_2=C_2(\lip_\sigma)>0$---neither 
	depending on $u_0$---such that uniformly
	for all real numbers $\lambda\ge1$, $k\ge 2$, and $t\geq t_0$,
	\[
		\E\left(\sup_{x\in\mathbb{T}}\sup_{s\in[t_0,t]}
		\left| u(s\,,x\,;\lambda)\right|^k\right)
		\le C_1^k k^{k/2}(1+|t-t_0|)^{(\varepsilon k +2)/2}
		\exp\left(\frac{C_2k^3\lambda^4t}{\varepsilon^2}\right)
		\|u_0\|_{L^\infty}^{k\varepsilon}\|u_0\|_{L^1}^{k(1-\varepsilon)}.
	\]
\end{proposition}

For us, the key feature of the preceding formula is the particular
way in which the expectation on the left-hand side is controlled by
the $L^1$ and $L^\infty$ norms of $u_0$ on the right. Still, we do have to be somewhat
careful about the other intervening constants in order to be sure that
they are not too large for our later use [they fortunately are not].

We will use Proposition \ref{pr:improved} and the related Proposition 
\ref{pr:improved:x} in the following way.  First we shift time so that we 
can replace $u_0$ by $u(t\,,\cdot\,;\lambda)$ and replace 
$u(t\,,\cdot\,;\lambda)$ by $u(t+h\,,\cdot\,;\lambda)$.  Then, by using our 
propositions, we can control $u(t+h\,,\cdot\,;\lambda)$ by the product of 
$\|u(t\,,\cdot\,;\lambda)\|_{L^\infty}$ to a small power, 
$\|u(t\,,\cdot\,;\lambda)\|_{L^1}$ to a large power, and by $\exp(-Ch)$ for 
some positive constant $C$.  In fact, we would rather have a negative 
exponential involving $t+h$, that is, $\exp\{-C(t+h)\}$.  To move from $h$ to 
$t+h$, we let $h$ be a multiple of $t$.  But we still need a negative 
exponent.  Proposition \ref{pr:M:small} shows that with high probability, 
$\|u(t\,,\cdot\,;\lambda)\|_{L^1}$ declines exponentially fast in $t$, and 
hence also in $t+h$.  We also have to deal with 
$\|u(t\,,\cdot\,;\lambda)\|_{L^\infty}$ raised to a small power.  But here 
we can use Propositions \ref{pr:improved} and \ref{pr:improved:x} once more, 
and the small power of $\|u(t\,,\cdot\,;\lambda)\|_{L^\infty}$ means that we 
have introduced a slowly-growing exponential $\exp(ct)$, which is comparable 
to $\exp\{c'(t+h)\}$ for a small constant $c'$.  We will see that the negative 
exponential wins out, with the result that $\|u(t\,,\cdot\,;\lambda)\|_\infty$ is small 
with high probability.

The proof of Proposition \ref{pr:improved} hinges on a series of 
intermediary results, some of which imply Proposition \ref{pr:simultaneous} as well.  
We will use the mild form \eqref{mild} to estimate $u(t\,,x\,;\lambda)$.
Our first technical result is an elementary interpolation
fact about the heat semigroup $\{P_t\}_{t\ge0}$, defined earlier in \eqref{P_tf}.
This result will allow us to estimate $P_tu_0$, the first term on the right side of \eqref{mild}.
Then we will use an argument related to Gronwall's lemma to estimate the second term
$\mathcal{I}_t(x\,;\lambda)$ on the right side of \eqref{mild}.  In fact, $\mathcal{I}_t(x\,;\lambda)$
is an integral containing terms which also involve the heat semigroup.

\begin{lemma}\label{lem:semigp}
	For every $t>0$ and $\varepsilon\in(0\,,1)$,
	\[
		\| P_tu_0\|_{L^\infty} \le 2\left( t^{-1/2}\vee 1\right)^{1-\varepsilon}
		\|u_0\|_{L^\infty}^\varepsilon
		 \|u_0\|_{L^1}^{1-\varepsilon}.
	\]
\end{lemma}

\begin{proof}
	We first observe that
	\begin{equation}\label{Ptu0:UB}
		\| P_tu_0\|_{L^\infty} \le \min\left( \|u_0\|_{L^\infty}\,, 
		2\left[ t^{-1/2}\vee 1\right] \|u_0\|_{L^1}\right).
	\end{equation}
	Indeed, since the semigroup $\{P_t\}_{t\ge0}$ is conservative, we clearly have
	$(P_tu_0)(x)\le \|u_0\|_{L^\infty}$
	for every $x\in\mathbb{T}$. And
	Lemma \ref{lem:c_l} below implies that
	$(P_tu_0)(x) \le 2 (t^{-1/2}\vee 1)\|u_0\|_{L^1}$
	for every $x\in\mathbb{T}$.
	Now that we have verified \eqref{Ptu0:UB} we deduce the lemma  
	from \eqref{Ptu0:UB} and the elementary fact that
	$\min(a\,,b) \le a^\varepsilon b ^{1-\varepsilon}$ for every $a,b>0$
	and $\varepsilon\in(0\,,1)$.
\end{proof}

Next we establish an improvement to Proposition \ref{pr:moment:UB}.
The following is indeed an improvement in the sense that it shows
how one can control the moments of the solution to \eqref{SHE-lambda} by
using both the $L^\infty$ and the $L^1$ norms of the initial data,
and not just the $L^\infty$ norm of $u_0$. This added improvement
does cost a little at small times. This latter fact is showcased by the appearance
of a negative power of $t$ in the following.

\begin{proposition}\label{pr:moment:UB:improved}
	Let $c:=208\sqrt 2\approx 294.2$. Then,
	for all real numbers $k\ge 2$, $\varepsilon\in(0\,,1)$, and $\lambda>0$
	that satisfy $k\lambda^2 \ge \varepsilon(c\lip^2_\sigma)^{-1},$
	the following holds uniformly for all $t>0$:
	\[
		\sup_{x\in\mathbb{T}} \E\left( |u(t\,,x\,;\lambda)|^k\right) \le \frac{4^k}{%
		t^{k(1-\varepsilon)/2}}\exp\left(\frac{c^2}{\varepsilon^2}k^3\lambda^4
		\lip_\sigma^4 t\right) \|u_0\|_{L^\infty}^{k\varepsilon} 
		\|u_0\|_{L^1}^{k(1-\varepsilon)}.
	\]
\end{proposition}

\begin{proof}	
	Let $\{u^{(n)}\}_{n=0}^\infty$ be the Picard approximants
	of $u$ (see \eqref{mild:1}). Thanks to Lemma \ref{lem:semigp},
	we can now write the following variation of \eqref{u:u:I}:
	For all integers $n\ge 0$
	and real numbers  $k\in[2\,,\infty)$, $t>0$, and $x\in\mathbb{T}$,
	\begin{equation}\label{u:u:I:n}
		\left\|u^{(n+1)}_t(x\,;\lambda)\right\|_k \le  
		2\left( t^{-(1-\varepsilon)/2}\vee 1\right)
		\|u_0\|_{L^\infty}^\varepsilon
		\|u_0\|_{L^1}^{1-\varepsilon}+ \left\|\cI^{(n)}_t(x\,;\lambda)\right\|_k.
	\end{equation}
	The latter quantity is estimated in \eqref{BOPO}. If we use that estimate
	in \eqref{u:u:I:n}, then the elementary inequality,
	$(a+b)^2\le 2a^2+2b^2$, valid for all $a,b\in\R$, yields the following:
	\begin{align}\label{mot}
		&\left\| u^{(n+1)}_t(x\,;\lambda)\right\|_k^2\\\notag
		&\le 8 \left( t^{-(1-\varepsilon)} \vee 1
			\right) \|u_0\|_{L^\infty}^{2\varepsilon} \|u_0\|_{L^1}^{2(1-\varepsilon)}
			+8k\lambda^2\lip_\sigma^2
			\int_0^t\d s\int_{-1}^1\d y\ \left[ p_{t-s}(x\,,y)\right]^2
			\left\|u^{(n)}_s(y\,;\lambda)\right\|_k^2\\\notag
		&\le 8\left( t^{-(1-\varepsilon)} \vee 1
			\right) \|u_0\|_{L^\infty}^{2\varepsilon} \|u_0\|_{L^1}^{2(1-\varepsilon)}
			+16k\lambda^2\lip_\sigma^2
			\int_0^t\sup_{y\in\mathbb{T}}
			\left\|u^{(n)}_s(y\,;\lambda)\right\|_k^2
			\left( \frac{1}{\sqrt{t-s}}+1\right)\d s.
	\end{align}
	We have appealed to \eqref{ppp} in the last line.
	The preceding motivates us to consider the temporal functions,
	$U^{(0)},U^{(1)},\ldots$, defined via
	\[
		U^{(n)}(t) := \sup_{x\in\mathbb{T}}\left\| 
		u^{(n)}_t(x\,;\lambda)\right\|_k^2\qquad[t\ge0],
	\]
	in order to obtain a recursive inequality.
	We can  see immediately from \eqref{mot} that, for all $n\ge0$ and $t>0$,
	\begin{align}\label{rec}
		&U^{(n+1)}(t) \\\notag
		&\le 8 \left( t^{-(1-\varepsilon)} \vee 1
			\right) \|u_0\|_{L^\infty}^{2\varepsilon} \|u_0\|_{L^1}^{2(1-\varepsilon)}
			+ 16k\lambda^2\lip_\sigma^2\int_0^t\frac{U^{(n)}(s)}{\sqrt{t-s}}\, \d s
			+ 16k\lambda^2\lip_\sigma^2\int_0^t U^{(n)}(s)\,\d s.
	\end{align}
	In order to understand this recursion more deeply, 
	let us first note that 
	\begin{equation}\label{BOP}
		\sup_{t>0} \left[\frac{t^{1-\varepsilon}}{\e^{\beta t}}
		\left( t^{-(1-\varepsilon)} \vee 1\right) \right]
		\le 1\vee \frac{1}{\beta^{1-\varepsilon}}\qquad\text{for all $\beta>0$}.
	\end{equation}
	This is true simply because  $t^{1-\varepsilon} \e^{-t} \le 1$ for all
	$t\ge0$.
	Therefore, we may define
	\[
		\cU^{(m)}(\beta) := \sup_{t\ge0} \left[ \frac{t^{1-\varepsilon}}{
		\e^{\beta t}}\, U^{(m)}(t)\right]
		\qquad[\beta>0,\, m\ge0],
	\]
	in order to deduce the following recursive inequality from \eqref{rec}
	and \eqref{BOP}:
	\[
		\cU^{(n+1)}(\beta)
		\le 8\left(1\vee \frac{1}{\beta^{1-\varepsilon}}
		\right)\|u_0\|_{L^\infty}^{2\varepsilon} \|u_0\|_{L^1}^{2(1-\varepsilon)}
		+16k\lambda^2\lip_\sigma^2 \left[C(\beta)
		+ D(\beta)\right]
		\cU^{(n)}(\beta);
	\]
	where we have defined, for all $0<\varepsilon<1$ and $\beta>0$,
	\[
		C(\beta) :=
		\sup_{t\ge0}\left[\int_0^t \left(\frac ts\right)^{1-\varepsilon}
		\frac{\e^{-\beta(t-s)}}{\sqrt{t-s}}\,\d s\right]
		\text{ and }
		D(\beta) :=
		\sup_{t\ge0}\left[\int_0^t \left(\frac ts\right)^{1-\varepsilon}
		\e^{-\beta(t-s)}\,\d s\right].
	\]
	It is possible to check [see Lemma \ref{lem:IJ} of
	the appendix] that
	\[
		C(\beta)+ D(\beta)
		\le \frac{13}{\varepsilon\sqrt\beta}\qquad\text{for all
		$\varepsilon\in(0\,,1)$ and $\beta\ge1$}.
	\]
	Thus, we obtain the recursive inequalities,
	\[
		\cU^{(n+1)}(\beta)
		\le 8\|u_0\|_{L^\infty}^{2\varepsilon} \|u_0\|_{L^1}^{2(1-\varepsilon)}
		+\frac{208k\lambda^2\lip_\sigma^2}{\varepsilon\sqrt\beta}
		 \cU^{(n)}(\beta),
	\]
	valid for all integers $n\ge0$, and reals $\beta\ge1$ and $\varepsilon\in(0\,,1)$.
	We can replace $\beta$ with
	\[
		\beta_* := \frac{173056 k^2\lambda^4
		\lip_\sigma^4}{\varepsilon^2},
	\]
	in order to see that for all $n\ge0$,
	\begin{equation}\label{boop}
		\cU^{(n+1)}(\beta_*) \le 
		 8\|u_0\|_{L^\infty}^{2\varepsilon} \|u_0\|_{L^1}^{2(1-\varepsilon)}+ 
		\tfrac12\cU^{(n)}(\beta_*),
	\end{equation}
	for all $n\ge0$,
	provided that $\beta_*\ge 1$.
	Note that
	\[
		\cU^{(0)}(\beta) =  \sup_{t\ge0} \left[ t^{1-\varepsilon}
		\e^{-\beta t}\right]\|u_0\|_{L^\infty}^2\le\|u_0\|_{L^\infty}^2,
	\]
	for every $\varepsilon\in(0\,,1)$ and $\beta\ge1$. In particular, 
	\eqref{boop} implies that: (i)
	$\sup_{n\ge 0}\cU^{(n)}(\beta_*) <\infty$; 
	and (ii) For all $n\ge0$, and provided that $\beta_*\ge 1$, 
	\[
		\limsup_{n\to\infty}\,\cU^{(n+1)}(\beta_*) 
		\le 16\|u_0\|_{L^\infty}^{2\varepsilon} \|u_0\|_{L^1}^{2(1-\varepsilon)},
	\]
	The left-hand side is greater than or equal to $t^{1-\varepsilon}\e^{-\beta_* t}
	\| u(t\,,x\,;\lambda)  \|_k^2$ uniformly for all $t>0$. This is
	thanks to Fatou's lemma and \eqref{Picard}. Therefore,
	for all $x\in\mathbb{T}$ and $t\ge0$,
	\[
		\left\| u(t\,,x\,;\lambda) \right\|_k^2 \le 16 t^{-1+\varepsilon}\e^{\beta_*t}\cdot
		\|u_0\|_{L^\infty}^{2\varepsilon} \|u_0\|_{L^1}^{2(1-\varepsilon)},
	\]
	provided that $k$ is large enough to ensure that
	$\beta_*\ge 1$. This is  equivalent to the assertion of the proposition.
\end{proof}

For our next technical result, let us recall the random field $\cI$ from \eqref{I}. 

\begin{lemma}\label{lem:modulus:x}
	Let $c:=208\sqrt 2\approx 294.2$. 
	For every $\varepsilon\in(0\,,1)$ and $\delta\in(0\,,1)$ there
	exists a finite constant $C=C(\varepsilon\,,\delta\,,\lip_\sigma)>0$---not 
	depending on $u_0$---such that uniformly
	for all real numbers $\lambda>0$, $x,y\in\mathbb{T}$, $k\ge 2$, and $t>0$
	that satisfy $k\lambda^2 \ge \varepsilon(c\lip^2_\sigma)^{-1},$
	\[
		\E\left(\left| \cI_t(x\,;\lambda)-\cI_t(y\,;\lambda) \right|^k\right)
		\le C^kk^{k/2}|x-y|^{\delta k/2} \max\left\{1\,, \frac{1}{t^{k/2}} \right\}
		\exp\left(\frac{c^2}{\varepsilon^2}
		k^3\lambda^4\lip_\sigma^4t\right)
		\|u_0\|_{L^\infty}^{k\varepsilon}\|u_0\|_{L^1}^{k(1-\varepsilon)}.
	\]
\end{lemma}

\begin{proof}
	We apply the Burkholder-Davis-Gundy-type inequality, as in the proof of Proposition \ref{pr:modulus:UB},
	in order to see that
	\begin{align*}
		&\left\| \cI_t(x\,;\lambda) - \cI_t(y\,;\lambda)\right\|_k^2
			\le 4k\lambda^2\int_0^t\d s\int_{-1}^1\d z\ 
			\left[ p_{t-s}(x\,,z) - p_{t-s}(y\,,z)\right]^2
			\left\| \sigma(u(s\,,z\,;\lambda))\right\|_k^2\\
		&\le 4k\lambda^2\lip_\sigma^2\int_0^t\d s\int_{-1}^1\d z\ 
			\left[ p_{t-s}(x\,,z) - p_{t-s}(y\,,z)\right]^2
			\left\| u(s\,,z\,;\lambda)\right\|_k^2\\
		&\le 64k\lip_\sigma^2\exp\left(\frac{2c^2}{\varepsilon^2}
			k^2\lambda^4\lip_\sigma^4t\right)
			\|u_0\|_{L^\infty}^{2\varepsilon}
			\|u_0\|_{L^1}^{2(1-\varepsilon)}\cdot\int_0^t
			\frac{\d s}{s^{1-\varepsilon}}\int_{-1}^1\d z\ 
			\left[ p_{t-s}(x\,,z) - p_{t-s}(y\,,z)\right]^2.
	\end{align*}
	The final inequality is a consequence of Proposition 
	\ref{pr:moment:UB:improved}, which is why we need the condition
	$k\lambda^2 \ge \varepsilon(c\lip^2_\sigma)^{-1}$.
	Apply Lemma \ref{lem:p-p:x:2}
	to see that
	\begin{align*}
		\int_0^t\frac{\d s}{s^{1-\varepsilon}}\int_{-1}^1\d z\ 
			\left[ p_{t-s}(x\,,z) - p_{t-s}(y\,,z)\right]^2
			&\le C'|x-y|^{\delta}\int_0^t
			\frac{\d s}{s^{1-\varepsilon}\times \left( 
			(t-s)^{(\delta+1)/2}\wedge(t-s)^{\delta/2}\right)}\\
		&\le C''\frac{|x-y|^\delta}{t^{\delta/2-\varepsilon}} 
			\max\left\{1\,, 1/\sqrt{t} \right\}\\
		&\le C''|x-y|^\delta\max\left\{1\,, 1/t \right\},
	\end{align*}
	where $C'$ and $C''$ are finite constants that depend only on 
	$\varepsilon$ and $\delta$. 
	The second inequality can be obtained by split the integral into 
	$\int_0^{t/2}\cdots\, \d s$ and $\int_{t/2}^{t}\cdots\,\d s$.
	The first integral is less than 
	\[
		\int_{0}^{t/2}\frac{\d s}{s^{1-\varepsilon}
		\min\left((t/2)^{(\delta+1)/2}\,, (t/2)^{\delta/2}\right)},
	\]
	and the second one is less than the same bound by a similar argument.
	The last inequality above comes from the fact that 
	$(\delta/2)-\varepsilon+(1/2)<1$ for all $\delta \in (0\,,1)$ and 
	$\varepsilon \in (0\,,1)$, so that $(1/t)^{\delta/2-\varepsilon+1/2} \leq 1/t$ for $t<1$. 
	We combine the preceding two displays to conclude the proof of the proposition.
\end{proof}

We can combine Lemmas \ref{lem:P_th(x)-P_th(y)} and \ref{lem:modulus:x} 
together with \eqref{mild} in order to  deduce the following.

\begin{proposition}\label{pr:modulus:x}
	Let $c:=208\sqrt 2\approx 294.2$. 
	For every $\varepsilon\in(0\,,1)$ and $\delta\in(0\,,1)$ there
	exists a finite constant $C=C(\varepsilon\,,\delta\,,\lip_\sigma)>0$---not 
	depending on $u_0$---such that uniformly
	for all real numbers $\lambda>0$, $x,y\in\mathbb{T}$, $k\ge 2$, and $t>0$
	that satisfy $k\lambda^2 \ge \varepsilon(c\lip^2_\sigma)^{-1},$
	\begin{align*}
		&\E\left(\left| u(t\,,x\,;\lambda) - u(t\,,y\,;\lambda) \right|^k\right)\\
		&\le C^kk^{k/2}\left(|x-y|^{\delta k/2}+
			|x-y|^{\varepsilon k/2}\right) \max\left\{1\,, \frac{1}{t^{k/2}}\right\}  
			\exp\left(\frac{c^2}{\varepsilon^2}
			k^3\lambda^4\lip_\sigma^4t\right)
			\|u_0\|_{L^\infty}^{k\varepsilon}\|u_0\|_{L^1}^{k(1-\varepsilon)}.
	\end{align*}
\end{proposition}

The preceding is a moment continuity result about $x\mapsto u(t\,,x\,;\lambda)$.
The following matches that result with a moment continuity estimate for
$t\mapsto u(t\,,x\,;\lambda)$.

\begin{proposition}\label{pr:modulus:t}
	Let $c:=208\sqrt 2\approx 294.2$. 
	For all real numbers $k\ge 2$, $\lambda>0$, 
	$\delta\in(0\,,\nicefrac14)$, $\varepsilon\in(0\,,1)$, and 
	$t>2\delta$ that satisfy $k\lambda^2 \ge \varepsilon(c\lip^2_\sigma)^{-1},$
	\begin{align*}
		\E\left(\left|u(t+\delta\, ,x\,;\lambda) - u(t\,,x\,;\lambda)\right|^k\right)
		\le 74^k \left[\frac{\delta^\varepsilon}{t^{1+\varepsilon}}+  
		\frac{\lambda^2 k\sqrt{\delta}\, \lip_\sigma^2}{%
		\varepsilon t^{1-\varepsilon}}\right]^{k/2}
		\exp\left(\frac{c^2}{\varepsilon^2}k^3\lambda^4
		\lip_\sigma^4 t\right) \cdot
		\|u_0\|_{L^\infty}^{k\varepsilon} 
		\|u_0\|_{L^1}^{k(1-\varepsilon)}.
	\end{align*}
\end{proposition}

\begin{proof}
	In accord with \eqref{mild} we can write
	\begin{equation}\label{I:T1:T2}
		\left\| \cI_{t+\delta}(x\,;\lambda) - \cI_t(x\,;\lambda)\right\|_k^2
		\le 2\lambda^2\left(\| T_1\|_k^2 + \|T_2 \|_k^2\right),
	\end{equation}
	where
	\begin{align*}
		T_1 & := \int_{(t,t+\delta)\times\mathbb{T}}
			p_{t-s+\delta}(x\,,z)\sigma(u(s\,,z\,;\lambda))
			\, W(\d s\,\d z),\\
		T_2 &:= \int_{(0,t)\times\mathbb{T}} 
			\left[p_{t-s+\delta}(x\,,z)-p_{t-s}(x\,,z)\right]
			\sigma(u(s\,,z\,;\lambda))
			\, W(\d s\,\d z).
	\end{align*}
	Now we apply the Burkholder-Davis-Gundy-type inequality, as in the proof of Proposition \ref{pr:modulus:UB},
	in order to see that
	\begin{align*}
		\|T_1\|_k^2 &\le 4k\lip_\sigma^2
			\int_t^{t+\delta}\d s\int_{-1}^1\d z\
			\left[ p_{t-s+\delta}(x\,,z)\right]^2\left\| u(s\,,z\,;\lambda)\right\|_k^2\\
		&\le \frac{64k \lip_\sigma^2}{%
			t^{1-\varepsilon}}\exp\left(\frac{2c^2}{\varepsilon^2}k^2\lambda^4
			\lip_\sigma^4 t\right) \|u_0\|_{L^\infty}^{2\varepsilon} 
			\|u_0\|_{L^1}^{2(1-\varepsilon)}\int_t^{t+\delta}\d s\int_{-1}^1\d z\
			\left[ p_{t-s+\delta}(x\,,z)\right]^2;
	\end{align*}
	consult Proposition \ref{pr:moment:UB:improved} for the last line.
	We appeal first to the semigroup property of $p_t$ and then to 
	Lemma \ref{lem:c_l} below in order to see from the bound $\delta\in(0\,,1/4)$ that
	\[
		\int_t^{t+\delta}\d s\int_{-1}^1\d z\
		\left[ p_{t-s+\delta}(x\,,z)\right]^2 \le 2\int_0^\delta\left(\frac{1}{\sqrt s}+1\right)
		\d s\le 4\sqrt{\delta}+2\delta<5\sqrt\delta,
	\]
	whence it follows that
	\begin{equation}\label{eq:T1:t}
		\|T_1\|_k^2 
		\le \frac{320k \lip_\sigma^2}{%
		t^{1-\varepsilon}}\exp\left(\frac{2c^2}{\varepsilon^2}k^2\lambda^4
		\lip_\sigma^4 t\right) \sqrt{\delta}\cdot
		\|u_0\|_{L^\infty}^{2\varepsilon} 
		\|u_0\|_{L^1}^{2(1-\varepsilon)}.
	\end{equation}
	Similarly, we have
	\begin{equation}\label{eq:T2:t}\begin{split}
		\|T_2\|_k^2 &\le 4k\lip_\sigma^2\int_0^t\d s\int_{-1}^1\d z\
			\left[p_{t-s+\delta}(x\,,z)-p_{t-s}(x\,,z)\right]^2
			\| u(s\,,z\,;\lambda)\|_k^2\\
		& \le 64k \lip_\sigma^2\exp\left(\frac{2c^2}{\varepsilon^2}k^2\lambda^4
			\lip_\sigma^4 t\right) \|u_0\|_{L^\infty}^{2\varepsilon} 
			\|u_0\|_{L^1}^{2(1-\varepsilon)}\times Q_t(\delta),
	\end{split}\end{equation}
	where
	\[
		Q_t(\delta) := \int_0^t\frac{\d s}{s^{1-\varepsilon}}\int_{-1}^1\d z\
		\left[p_{t-s+\delta}(x\,,z)-p_{t-s}(x\,,z)\right]^2.
	\]
	Lemma \ref{lem:p-p:t:1} below tells us that
	\begin{align*}
		Q_t(\delta)&\le \sqrt{\frac{\pi}{2}}\int_0^t
			\frac{\d s}{s^{1-\varepsilon}\sqrt{t-s}}\min\left(1\,,\frac{\delta}{t-s}\right)\\
		&=  \sqrt{\frac{\pi}{2}} t^{\varepsilon-(1/2)}\int_0^1
			\frac{\d r}{r^{1-\varepsilon}\sqrt{1-r}}\min\left(1\,,\frac{\delta/t}{1-r}\right).
	\end{align*}
	If $t>2\delta$, then we write
	\[
		Q_t(\delta) \le \sqrt{\frac{\pi}{2}} \delta 
		t^{\varepsilon-(3/2)}\int_0^{1-(\delta/t)}
		\frac{\d r}{r^{1-\varepsilon}(1-r)^{3/2}}
		+ \sqrt{\frac{\pi}{2}} t^{\varepsilon-(1/2)}\int_{1-(\delta/t)}^1
		\frac{\d r}{r^{1-\varepsilon}\sqrt{1-r}}.
	\]
	We can write the first integral as
	\begin{align*}
		\int_0^{1/2}\frac{\d r}{r^{1-\varepsilon}(1-r)^{3/2}} +
			\int_{1/2}^{1-(\delta/t)}\frac{\d r}{r^{1-\varepsilon}(1-r)^{3/2}}
			&\le \frac{2^{(3/2)-\varepsilon}}{\varepsilon} + 
			2^{1-\varepsilon}\int_{\delta/t}^\infty
			\frac{\d r}{r^{3/2}}\\
		&= \frac{2^{(3/2)-\varepsilon}}{\varepsilon} + 2^{2-\varepsilon}
			(\delta / t)^{-1/2}\\
		&\le \frac{8}{\varepsilon}(\delta / t )^{-1/2}.
	\end{align*}
	[We have used the bound $t>2\delta$ in the last line.]
	And the second integral is bounded from above by
	\[
		\left( 1-\frac\delta t\right)^{-1+\varepsilon}\int_0^{\delta /t}\frac{\d r}{\sqrt r}
		= 2\left( 1-\frac\delta t\right)^{-1+\varepsilon}\sqrt{\delta/t}
		< 4\sqrt{\delta/t}.
	\]
	This yields
	\[
		Q_t(\delta) \le
		\frac{16}{\varepsilon}\cdot
		\frac{\sqrt{\delta}}{t^{1-\varepsilon}}.
	\]
	We may apply this inequality in \eqref{eq:T2:t} in order to see that
	\[
		\|T_2\|_k^2 \le \frac{1024k \lip_\sigma^2}{\varepsilon\cdot t^{1-\varepsilon}}
		\exp\left(\frac{2c^2}{\varepsilon^2}k^2\lambda^4
		\lip_\sigma^4 t\right)\sqrt\delta\cdot \|u_0\|_{L^\infty}^{2\varepsilon} 
		\|u_0\|_{L^1}^{2(1-\varepsilon)}.
	\]
	We combine this with \eqref{eq:T1:t} and \eqref{I:T1:T2} in order to deduce
	the following:
	\[
		\left\| \cI_{t+\delta}(x\,;\lambda) - \cI_t(x\,;\lambda)\right\|_k^2
		\le \lambda^2 \frac{2688k \lip_\sigma^2}{%
		t^{1-\varepsilon}}\exp\left(\frac{2c^2}{\varepsilon^2}k^2\lambda^4
		\lip_\sigma^4 t\right) \frac{\sqrt{\delta}}{\varepsilon}\cdot
		\|u_0\|_{L^\infty}^{2\varepsilon} 
		\|u_0\|_{L^1}^{2(1-\varepsilon)}.
	\]
	This and the argument in Lemma \ref{lem:P_th(x)-P_sh(x)} below together yield
	\begin{align*}
		\| u(t+\delta\,,x\,;\lambda) - u(t\,,x\,;\lambda)\|_k^2 \le
			2 \| (P_{t+\delta} u_0)(x) - (P_t u_0)(x)\|_k^2
			+ 2\left\| \cI_{t+\delta}(x\,;\lambda) - \cI_t(x\,;\lambda)\right\|_k^2&\\
		\le \left[\frac{8\delta^\varepsilon}{t^{1+\varepsilon}}
			+ \frac{5376k\sqrt{\delta}\, \lambda^2 \lip_\sigma^2}{%
			\varepsilon t^{1-\varepsilon}}\exp\left(\frac{2c^2}{\varepsilon^2}k^2\lambda^4
			\lip_\sigma^4 t\right)\right]\cdot
			\|u_0\|_{L^\infty}^{2\varepsilon} 
			\|u_0\|_{L^1}^{2(1-\varepsilon)}.&
	\end{align*}
	This easily implies the result.
\end{proof}

Before we derive Proposition \ref{pr:improved}---the main result
of this section---we pause and quickly establish Proposition \ref{pr:simultaneous}.

\begin{proof}[Proof of Proposition \ref{pr:simultaneous}]
	We can combine Propositions \ref{pr:modulus:UB}, \ref{pr:modulus:x},
	and \ref{pr:modulus:t} together with the Kolmogorov continuity theorem
	in order to see that $(t\,,x\,,\lambda)\mapsto u(t\,,x\,;\lambda)$ has
	a continuous modification on $(0\,,\infty)\times\R\times(0\,,\infty)$.
	The proofs of Propositions 
	\ref{pr:modulus:UB}, \ref{pr:modulus:x},
	and \ref{pr:modulus:t} also imply, implicitly, the fact that two quantities on the right-hand
	side of \eqref{mild}---viewed as random functions of $(t\,,x\,,\lambda)$---have 
	continuous modifications on $(0\,,\infty)\times\R\times(0\,,\infty)$. 
	It follows that \eqref{mild} holds for all $(t\,,x\,,\lambda)\in(0\,,\infty)\times\R\times(0\,,\infty)$
	off a single null set. This and a stochastic Fubini argument together imply the result.
\end{proof}

We are finally ready to prove Proposition \ref{pr:improved}. Before we
commence, however, it might be helpful to explicitly state the
following well-known chaining argument \cite{DKMNX}. It might help to recall
that an \emph{upright box} in $\R^N$ has the form
$\prod_{i=1}^N[a_i\,,b_i]$, where $a_i\le b_i$ are real numbers
[$1\le i\le N$].

\begin{proposition}\label{pr:KCT}
	Suppose $\{X(t)\}_{t\in T}$ is a real-valued stochastic process,
	where $T$ is a bounded upright box in $\R^N$ for some $N\ge 1$. Suppose
	also that there exists $Q\in(0\,,\infty)$ such that for every integer $K\ge 2$
	\[
		B_K := \E\left(|X(s)|^K\right)<\infty
		\quad\text{for some }s\in T\text{ and}\quad
		C_K := \sup_{\substack{s,t\in T\\s\neq t}}
		\E\left(\frac{|X(t)-X(s)|^K}{|t-s|^{KQ}}\right) <\infty,
	\]
	where $|\tau|$ denotes any one of the $\ell^p$-norms on $\tau\in\R^N$
	$[0<p<\infty]$. Then, there exists a finite constant $D$---%
	depending also on the diameter of $T$, $N$, $Q$ and $K$ with $QK>N$,
	---such that
	\[
		\E\Bigg(\sup_{\substack{s,t\in T\\s\neq t}}
		|X(t)-X(s)|^K\Bigg) \le D^K C_K\quad\text{and hence}\quad
		\E\left(\sup_{t\in T}|X(t)|^K\right) \le 2^K(B_K+D^K C_K).
	\]
\end{proposition}
	
\begin{proof}[Proof of Proposition \ref{pr:improved}]
	Combine Propositions
	\ref{pr:modulus:x}, \ref{pr:modulus:t}, and \ref{pr:KCT},
	all the time keeping track of the various [explicit] constants.
\end{proof}

Let us observe also  the following fixed-time result, which is
proved exactly as Proposition \ref{pr:improved} was, but without
the $t$-uniformity.

\begin{proposition}\label{pr:improved:x}
	There exists $\varepsilon_0=\varepsilon_0(\lip_\sigma)\in(0\,,1)$ such that 
	for every $\varepsilon\in(0\,,\varepsilon_0)$ there
	exist finite constants $C_1=C_1(\varepsilon\,,\lip_\sigma)>0$
	and $C_2=C_2(\lip_\sigma)>0$---not 
	depending on $u_0$---such that uniformly
	for all real numbers $\lambda\ge1$, $k\ge 2$ that satisfy 
	$k\lambda^2\geq\varepsilon(c\lip_\sigma)^{-1}$, and for every $t>0$,
	\[
		\E\left(\sup_{x\in\mathbb{T}}
		\left| u(t\,,x\,;\lambda)\right|^k\right)
		\le C_1^kk^{k/2}\left(1+\frac{1}{t^{k/2}}\right)
		\exp\left(\frac{C_2k^3\lambda^4t}{\varepsilon^2}\right)
		\|u_0\|_{L^\infty}^{k\varepsilon}\|u_0\|_{L^1}^{k(1-\varepsilon)}.
	\]
\end{proposition}

\begin{proof}[Proof of Proposition \ref{pr:improved:x}]
	Combine Propositions \ref{pr:moment:UB:improved}, 
	\ref{pr:modulus:x} and \ref{pr:KCT}, all the time
	keeping track of the various [explicit] constants.
\end{proof}

\section{Proof of Theorem \ref{th:lambda:UB}}\label{sec:Pf:th:lambda:UB}
Define $\cF:=\{\cF_t^0\}_{t>0}$ denote the filtration of sigma-algebras
that is defined via
\[
	\cF_t^0 := \sigma\left\{ \int_{(0,t)\times\mathbb{T}}
	\phi(s\,,y)\, W(\d s\,\d y):\
	\phi\in L^2((0\,,t]\times\mathbb{T})\right\},
\]
for every $t\ge0$.
Let $\P^{u_0}$ denote the law of the process $\{u(t\,,x\,;\lambda)\}_{t\ge0,x\in\mathbb{T}}$,
conditional on the initial state being $u_0$. Then we can define
\begin{equation}\label{cF}
	\cF_t := \bigcap_{s>t} \overline{\cF_s^0}\qquad[t\ge0],
\end{equation}
where $\overline{\cF_s^0}$ denotes the completion of $\cF_s^0$ with respect to
the family $\{\P^{u_0}\}_{u_0\in L^\infty}$ of probability measures.
Intuitively speaking, the filtration $\cF:=\{\cF_t\}_{t\ge0}$ is the Brownian filtration
that corresponds to the infinite-dimensional Brownian motion
$t\mapsto W(t\,,\cdot)$.

It is well known, see \cite[Theorem 9.15, page 256]{DZ},
that the process $t\mapsto u(t\,,\cdot\,;\lambda)$ is a Markov process,
with values in $C(\mathbb{T})$, with respect to the filtration $\cF:=\{\cF_t\}_{t\ge0}$
and initial measures $\{\P^{u_0}\}_{u_0\in L^\infty}$. That is,
\[
	\E\left[\Phi(u(t+\tau\,,\cdot\,;\lambda))\mid \cF_t\right] = \E^{u_t}\left[ \Phi(u(\tau\,,\cdot\,;\lambda))\right]
	\qquad\text{a.s.,}
\]
for every bounded functional $\Phi:C(\mathbb{T})\to\R_+$, and all $t,\tau\ge0$,
where we have suppressed the notational dependence on $\lambda$ 
to keep the notation simple.
We can restate this fact as follows: Choose and fix $t\ge0$ and define
$v(\tau\,,x):= u(t+\tau\,,x\,;\lambda)$ for all $\tau\ge0$ and $x\in\mathbb{T}$. 
Then, conditioned on $\cF_t$, the random field
$\{v(\tau\,,x)\}_{\tau\ge0,x\in\mathbb{T}}$ solves the SPDE \eqref{SHE-lambda}
[in law], started at $v(0\,,x) := u(t\,,x\,;\lambda)$, where now the noise $W$ is replaced
by a Brownian sheet $W^{(t)}$ that is independent of $\cF_t$. 
In particular, we may appeal to Proposition \ref{pr:improved:x},
conditionally, as follows: 
There exists $\varepsilon_0=\varepsilon_0(\lip_\sigma)\in(0\,,\nicefrac12)$
such that for every $\varepsilon\in(0\,,\varepsilon_0)$
there exist finite constants $C_1=C_1(\varepsilon\,,\lip_\sigma)>0$
and $C_2=C_2(\lip_\sigma)>0$---not 
depending on $u_0$---such that uniformly
for all real numbers $\lambda\ge2$ and $t>0$ and $h\in (0,1)$,
\begin{equation}\label{eq:est:prior}
	\E\left(\left.\sup_{x\in\mathbb{T}}
	\left| u(t+h\,,x\,;\lambda)\right|^2\ \right|\, \mathscr{F}_t\right)
	\le \frac{C_1\e^{C_2\lambda^4 h/\varepsilon^2 }}{h}
	\|u(t\,,\cdot\,;\lambda)\|_{L^\infty}^{2\varepsilon}
	\|u(t\,,\cdot\,;\lambda)\|_{L^1}^{2(1-\varepsilon)},
\end{equation}
almost surely. 

Now consider the event,
\begin{equation}\label{A}
	\mathbf{A}(t\,;\lambda) :=\left\{\omega\in\Omega:\ \|u(t\,,\cdot\,;\lambda)
	(\omega)\|_{L^1}
	\le \|u_0\|_{L^1}\exp\left(-\frac{\lambda^2{\rm L}_\sigma^2 t}{8}\right) \right\}.
\end{equation}
According to Proposition \ref{pr:M:small} [with $\varepsilon:=\nicefrac12$],
\begin{equation}\label{P(A)}
	\P(\mathbf{A}(t\,;\lambda)) \ge 
	1- \e^{ -\lambda^2{\rm L}_\sigma^2t/64}.
\end{equation}
Also, \eqref{eq:est:prior} implies that
\begin{equation}\label{pre:Jensen}\begin{split}
	&\E\left(\sup_{x\in\mathbb{T}}
		\left| u(t+h\,,x\,;\lambda)\right|^2 ;\ \mathbf{A}(t\,;\lambda)
		\right)\\
	&\hskip1.2in\le \frac{C_1\e^{C_2
		\lambda^4 h/\varepsilon^2}\|u_0\|_{L^1}^{2(1-\varepsilon)}
		}{h}
		\exp\left(-\frac{(1-\varepsilon)
		\lambda^2{\rm L}_\sigma^2 t}{4}\right) 
		\E\left(\|u(t\,,\cdot\,;\lambda)\|_{L^\infty}^{2\varepsilon}\right).
\end{split}\end{equation}
Since $\varepsilon\in(0\,,\varepsilon_0)\subset(0\,,1)$, Jensen's inequality
shows that
\begin{equation}\label{eq:by:Jensen}
	\E\left(\|u(t\,,\cdot\,;\lambda)\|_{L^\infty}^{2\varepsilon}\right)
	\le \left| \E\left(\|u(t\,,\cdot\,;\lambda)\|_{L^\infty}^2\right)\right|^\varepsilon.
\end{equation}
Proposition \ref{pr:improved:x} implies the following [set $k:=2$
and $\varepsilon:=\nicefrac12$ in the statement of the proposition]:
There exists a positive and finite constant $C_3$ such that for all 
$t>0$ and $\lambda^2\ge(4c\lip_{\sigma})^{-1}$,
\[
	\E\left(\|u(t\,,\cdot\,;\lambda)\|_{L^\infty}^2\right) \le
	\frac{C_3\e^{C_3\lambda^4t}}{t}
	\|u_0\|_{L^\infty}\|u_0\|_{L^1}.
\]
We plug this estimate into \eqref{eq:by:Jensen}, and then appeal to \eqref{pre:Jensen},
and the fact that $\varepsilon<\varepsilon_0<\nicefrac12$,
in order to see that
\begin{equation}\label{oh:boy}
	\E\left(\sup_{x\in\mathbb{T}}
	\left| u(t+h\,,x\,;\lambda)\right|^2 ;\ \mathbf{A}(t\,;\lambda)\right)
	\le \frac{C_4\e^{C_5\lambda^4 [(h/\varepsilon^2 )+\varepsilon t]}
	\|u_0\|_{L^1}^{2-\varepsilon}\|u_0\|_{L^\infty}^\varepsilon
	}{h t^{\varepsilon}}
	\e^{-\lambda^2{\rm L}_\sigma^2 t/8},
\end{equation}
where $C_4:=C_1C_3^{\varepsilon}$ and $C_5:=\max(C_2\,,C_3)$. Note that the implied constants
do not depend on $(t\,,h\,,\lambda)$. 

We now specialize the preceding to the following choice of $\varepsilon$ and $h$:
\[
	\varepsilon := \frac{{\rm L}_\sigma^2}{32C_5\lambda^2}
	\quad\text{and}\quad
	h:=\frac{{\rm L}_\sigma^6 t}{(32C_5\lambda^2)^3}.
\]
This choice is permissible, provided that $\varepsilon<\varepsilon_0<\nicefrac12$;
since if $\lambda$ is large enough so that $\varepsilon<\varepsilon_0$ and $h<1$. 
Because $\varepsilon_0$ does not depend on $t$, it follows that for every $t>0$,
\[
	\limsup_{\lambda\uparrow\infty}\frac{1}{\lambda^2}\log
	\E\left(\sup_{x\in\mathbb{T}}
	\left| u\left(t+C_6\lambda^{-6}\,,x\,;\lambda\right)\right|^2 ;\ \mathbf{A}(t\,;\lambda)\right)
	\le -\frac{{\rm L}_\sigma^2 t}{16},
\]
where $C_6= {\rm L}_{\sigma}^6 (32C_5)^{-3}t$. By the Chebyshev inequality,
\[
	\limsup_{\lambda\uparrow\infty}\frac{1}{\lambda^2}
	\log \P\left\{ \sup_{x\in\mathbb{T}} u\left(t+C_6\lambda^{-6},x\,;\lambda\right) \ge 
	\exp\left(-\frac{{\rm L}_\sigma^2\lambda^2 t}{64}
	\right) ;\ \mathbf{A}(t\,;\lambda)\right\}\le-\frac{{\rm L}_\sigma^2t}{32}.
\]
It is easy to see, after a change of variables in the preceding quantitative
bounds [before we apply the limsup], that the preceding
holds also with $u(t\,,x\,;\lambda)$ in place of $u(t+C_6\lambda^{-6}\,,x\,;\lambda)$.
For otherwise, we simply replace $t$ by $t-C_6\lambda^{-6}$
in all of the formulas before we let $\lambda\uparrow\infty$. 
In this way, we can  combine the above estimate 
with \eqref{P(A)} in order to deduce the theorem.\qed


\section{Proof of Theorem \ref{th:decay:t}}

\begin{proof}[Proof of Theorem \ref{th:decay:t}: Upper bound]
	Throughout, we choose and hold $\lambda>0$ fixed. 
	
	The proof of \eqref{eq:est:prior} shows also the following variation, thanks to Proposition \ref{pr:improved}:
	There exists $\varepsilon_0=\varepsilon_0(\lip_\sigma)\in(0\,,\nicefrac12)$, 
	small enough, such that 
	for every $\varepsilon\in(0\,,\varepsilon_0)$ there
	exist finite constants $C_1=C_1(\varepsilon\,,\lip_\sigma)>0$
	and $C_2=C_2(\lip_\sigma)>0$---not 
	depending on $u_0$---such that uniformly
	for all real numbers $\lambda\ge2$ and $\eta\in(0\,,1)$ and $t\ge t_0:=1$,  
	\begin{align*}
		&\E\left(\left.\sup_{x\in\mathbb{T}}\sup_{h\in[1,\eta (t+1)+1]}
			\left| u(t+h\,,x\,;\lambda)\right|^2\, \right|\ \mathscr{F}_t\right)\\
		&\hskip1in\le C_1\left[1+\eta(t+1)\right]^{3/2} 
			\e^{C_2\lambda^4 (\eta (t+1)+1)/\varepsilon^2}
			\|u(t\,,\cdot\,;\lambda)\|_{L^\infty}^{2\varepsilon}
			\|u(t\,,\cdot\,;\lambda)\|_{L^1}^{2(1-\varepsilon)},
	\end{align*}
	almost surely.  We appeal to this bound with $\eta:=\varepsilon^3$ in order
	to see that for all real numbers $\lambda\ge2$, $\eta\in(0\,,1)$, and $t\ge t_0:=1$,
	\begin{align*}
		&\E\left(\left.\sup_{x\in\mathbb{T}}\sup_{h\in[1,\varepsilon^3 (t+1)+1]}
			\left| u(t+h\,,x\,;\lambda)\right|^2\, \right|\ \mathscr{F}_t\right)\\
		&\hskip1in\le C_1\left[
			1+\varepsilon^3(t+1)\right]^{3/2}\e^{C_2\lambda^4/\varepsilon^2}
			\e^{C_2\lambda^4 \varepsilon (t+1)}
			\|u(t\,,\cdot\,;\lambda)\|_{L^\infty}^{2\varepsilon}
			\|u(t\,,\cdot\,;\lambda)\|_{L^1}^{2(1-\varepsilon)},
	\end{align*}
	almost surely.  
	It follows from this inequality
	that, for the same set ${\bf A}(t\,;\lambda)$ as was defined
	in \eqref{A}, the following variation of \eqref{oh:boy} holds:
	\begin{align*}
		&\E\left(\sup_{x\in\mathbb{T}}\sup_{h\in[1,\varepsilon^3 (t+1)+1]}
			\left| u(t+h\,,x\,;\lambda)\right|^2 ;\ \mathbf{A}(t\,;\lambda)\right)\\
			&\le C_7\left[1+\varepsilon^3(t+1)\right]^{3/2} 
			\e^{2C_5\lambda^4 \varepsilon  t}
			\|u_0\|_{L^1}^{2-\varepsilon}\|u_0\|_{L^\infty}^\varepsilon
			\e^{-\lambda^2{\rm L}_\sigma^2 t/8}\\
		& = C_7\left[1+\varepsilon^3(t+1)\right]^{3/2} \exp\left( -\lambda^2 t\left[
			\frac{{\rm L}_\sigma^2}{8} - 2 C_5\lambda^2\varepsilon\right]\right)
			\|u_0\|_{L^1}^{2-\varepsilon}\|u_0\|_{L^\infty}^\varepsilon,
	\end{align*}
	provided, additionally, that $0<\varepsilon<\varepsilon_0$; here, 
	$C_7=C_7(\varepsilon\,,\lip_\sigma,\lambda)$ is a positive and finite constant,
	and $C_5$ is the same constant that appeared in \S\ref{sec:Pf:th:lambda:UB}.
	We use the preceding inequality with the following special choice:
	\[
		\varepsilon := \min\left( \frac{\varepsilon_0}{2}\,,\frac{{\rm L}_\sigma^2}{32
		C_5 \lambda^2}\right).
	\]
	For this particular choice of $\varepsilon$, we have
	\[
		\E\left(\sup_{x\in\mathbb{T}}\sup_{s\in[t+1,(1+\varepsilon^3) (t+1)]}
		\left| u(s\,,x\,;\lambda)\right|^2 ;\ \mathbf{A}(t\,;\lambda)\right)
		\le  C_8\left[1+\varepsilon^3(t+1)\right]^{3/2} 
		\e^{-{\rm L}_\sigma^2 \lambda^2 (t+1)/16},
	\]
	uniformly for all $t\ge 1$, where $C_8:=C_7\exp(\lambda^2 {\rm L}_\sigma^2 /16)
	\|u_0\|_{L^1}^{2-\varepsilon}\|u_0\|_{L^\infty}^\varepsilon$ is a finite
	constant that does not depend on $t$. Define
	\[
		\mu := \log(1+\varepsilon^3),
	\]
	to see that for large integers $N$, and replace $t$ by $\exp(N\mu)-1$
	\[
		\E\left(\sup_{x\in\mathbb{T}}\sup_{s\in[\exp(N\mu),\exp([N+1]\mu)]}
		\left| u(s\,,x\,;\lambda)\right|^2 ;\ \mathbf{A}(\e^{N\mu}-1\,;\lambda)\right)
		\le  C_8\left[1+\varepsilon^3\e^{N\mu}\right]^{3/2}
		\e^{-{\rm L}_\sigma^2 \lambda^2 \e^{N\mu}/16}.
	\]
	In particular, Chebyshev's inequality shows that for all $\rho>0$,
	\begin{align*}
		&\P\left\{ \sup_{x\in\mathbb{T}}\sup_{s\in[\exp(N\mu),\exp([N+1]\mu)]}
			\left| \frac{u(s\,,x\,;\lambda)}{\e^{- {\rm L}_\sigma^2 \lambda^2s/64}}
			\right| \ge \rho\,; 
			{\bf A}(\e^{N\mu}-1\,;\lambda)\right\}\\
		&\le \P\left\{ \sup_{x\in\mathbb{T}}\sup_{s\in[\exp(N\mu),\exp([N+1]\mu)]}
			\left|u(s\,,x\,;\lambda) \right| \ge
			\rho\e^{- {\rm L}_\sigma^2 \lambda^2 \e^{(N+1)\mu}/64}\,; 
			{\bf A}\left(\e^{N\mu}-1\,;\lambda\right)\right\}\\
		&\le \frac{C_8}{\rho^2}\left[1+\varepsilon^3\e^{N\mu}\right]^{3/2} 
			\exp\left(-\frac{ {\rm L}_\sigma^2 \lambda^2 \e^{N\mu}}{64}\right).
	\end{align*}
	Combine this estimate with \eqref{P(A)} to see that
	\begin{align*}
		&\sum_{N=1}^\infty
			\P\left\{ \sup_{x\in\mathbb{T}}\sup_{s\in[\exp(N\mu),\exp([N+1]\mu)]}
			\left| \frac{u(s\,,x\,;\lambda)}{\e^{- {\rm L}_\sigma^2 \lambda^2s/64}}
			\right| \ge \rho\right\}\\
		&\hskip.7in\le \frac{C_8}{\rho^2}\sum_{N=1}^\infty
			\left[1+\varepsilon^3\e^{N\mu}\right]^{3/2}
			\exp\left(-\frac{ {\rm L}_\sigma^2 \lambda^2 \e^{N\mu}}{64}\right)
			+\sum_{N=1}^\infty
			\exp\left( -\frac{\lambda^2{\rm L}_\sigma^2\left[
			\e^{N\mu}-1\right]}{64}\right)<\infty.
	\end{align*}
	We can conclude from this and the Borel--Cantelli that
	\begin{equation}\label{Lower:1/2}
		\lim_{t\to\infty}
		\sup_{x\in\mathbb{T}}
		\left|\frac{u(t\,,x\,;\lambda)}{\e^{-\lambda^2 {\rm L}_\sigma^2 t/64}}\right|=0,
	\end{equation}
	almost surely. This completes the first half of the proof of Theorem \ref{th:decay:t}.
\end{proof}

The proof of the lower bound of Theorem \ref{th:decay:t}
depends on the following large-deviations bound for sums of dependent Bernoulli
random variables. For a proof see Lemma 3.9 of Khoshnevisan, R\'ev\'esz, and Shi \cite{KRS}.

\begin{lemma}\label{lem:LD}
	Suppose $J_1,J_2,\ldots$ are $\{0\,,1\}$-valued random variables
	that satisfy the following for some non-random constant $q>0$:
	$\E ( J_{k+1}\mid J_1,\ldots,J_k ) \ge q$ for all $k\ge 1$, a.s.
	Then,
	\[
		\P\left\{  J_1+\cdots+J_n \le nq(1-\varepsilon) \right\} \le
		\exp\left(-\frac{nq \varepsilon^2}{2}\right)
		\qquad\text{for every $\varepsilon\in(0\,,1)$ and
		$n\ge1$.}
	\]
\end{lemma}

We now proceed with the derivation of the lower bound of Theorem \ref{th:decay:t}.

\begin{proof}[Proof of Theorem \ref{th:decay:t}: Lower bound]
	We appeal to a one-sided adaptation of a method of Mueller \cite{Mueller1}.
	Define $T_0:=0$ and then iteratively let
	\[
		T_{n+1} := \inf\left\{ t>T_n:\  \inf_{x\in\mathbb{T}} u(t\,,x\,;\lambda) < 
		\e^{-1}\inf_{x\in\mathbb{T}}
		u(T_n\,,x\,;\lambda)\right\}\qquad\text{for all $n\ge 0$},
	\]
	where $\inf\varnothing:=\infty$. We have already proved in \eqref{Lower:1/2} that 
	$\sup_{x\in\mathbb{T}}u(t\,,x\,;\lambda)\to 0$ a.s.\ as $t\to\infty$. Therefore,
	$T_n<\infty$ for all $n\ge 0$ almost surely. Moreover, the sample-function continuity 
	of $u$ shows that the $T_n$'s are stopping times with respect to the filtration
	$\cF:=\{\cF_t\}_{t\ge0}$, defined earlier in \eqref{cF}. We may apply the 
	strong Markov property of $u$, with respect to $\cF$, at the stopping time $T_n$
	in order to see that for all integers $n\ge0$,
	\[
		\P\left( \left. T_{n+1}- T_n \le \tau \ \right|\, \cF_{T_n}\right)
		\le \P\left( \left.\inf_{x\in\mathbb{T}}\inf_{t\in[0,\tau]} v^{(n)}(t\,,x) \le \e^{-1}
		\inf_{x\in\mathbb{T}}v^{(n)}_0(x) \ \right|\, v^{(n)}_0\right),
	\]
	a.s., where $v^{(n)}:=\{v^{(n)}(t\,,x)\}_{t\ge0,x\in[-1,1]}$ solves \eqref{SHE-lambda} starting from
	the random initial profile $v^{(n)}_0(x) := v^{(n)}(0\,,x)=u(T_n\,,x\,;\lambda)$. By the very definition
	of the stopping time $T_n$, and since $T_n$ is finite a.s., 
	\[
		v^{(n)}_0(x) = u(T_n\,,x\,;\lambda)\ge \e^{-1} \inf_{y\in\mathbb{T}}u(T_{n-1}\,,y\,;\lambda)
		\ge \cdots \ge \e^{-n} \inf_{y\in\mathbb{T}}u_0(y) =: \e^{-n}\underline{u}_0,
	\]
	a.s.\ for every $x\in\mathbb{T}$, and with identity for some $x\in\mathbb{T}$ a.s.  
	Because $\underline{u}_0>0$ [see \eqref{u_0>0}],
	it follows from a comparison theorem \cite{ChenKim,Mueller1,Shiga} that
	$v^{(n)}(t\,,x) \ge w^{(n)}(t\,,x)$ for all $t\ge0$ and $x\in\mathbb{T}$ a.s., where
	$w^{(n)}$ solves \eqref{SHE-lambda} [for a different Brownian sheet]
	starting from $w^{(n)}_0(x):= w^{(n)}(0\,,x)= \e^{-n}\underline{u}_0$.
	In particular, for all integers $n\ge0$ and reals $\tau\in(0\,,1)$,
	\begin{align*}
		\P\left( T_{n+1}- T_n \le\tau \mid \cF_{T_n}\right)
			&\le \P\left\{\inf_{x\in\mathbb{T}}\inf_{t\in[0,\tau]} w^{(n)}(t\,,x) \le 
			\e^{-1}w^{(n)}_0(x) \equiv \e^{-1-n} 
			\underline{u}_0\right\}\\
		&\le \P\left\{ \sup_{x\in\mathbb{T}}\sup_{t\in[0,\tau]}\left|
			w^{(n)}(t\,,x) - w^{(n)}_0(x)\right| \ge\e^{-1-n}\underline{u}_0
			\right\}\\
		&\le \e^2\E\left( \sup_{x\in\mathbb{T}}\sup_{t\in[0,\tau]}\left|
			\frac{w^{(n)}(t\,,x)}{w^{(n)}_0(0)}-1\right|^2\right).
	\end{align*}
	Now, $ z^{(n)}(t\,,x) := w^{(n)}(t\,,x)/w^{(n)}_0(0)$ solves
	\eqref{SHE-lambda}, started at $ z^{(n)}_0(x):=z^{(n)}(0\,,x)= 1$, with $\sigma$ replaced by
	\[
		\sigma^{(n)}(a) := \frac{\sigma\left( w_0(0) a\right)}{w_0(0)}
		= \frac{\e^n}{\underline{u}_0}
		\sigma\left(\frac{\underline{u}_0 a}{\e^n}\right)\qquad[a\in\R].
	\]
	We may observe that the Lipschitz constant of $\sigma^{(n)}$ is
	$\lip_\sigma$, uniformly for all $n\ge 0$. In this way we find that
	there exists a finite constant $C$ uniformly in $n$ and $K$ such that
	\[
		\sup_{\substack{(t,x),(s,y)\in (0,\tau)\times\mathbb{T}\\s\neq t, x\neq y}}
		\E\left( \frac{ \left| z^{(n)}(t\,,x) - z^{(n)}(s\,,y) \right |^ K}{ \left | 
		(t\,,x)-(s\,,y) \right |^{K/4}}\right)
		\le (CK)^{K/2} \lip_\sigma^K \lambda^K \e^{C K \lambda^4 
		\lip_{\sigma}^4 \tau},
	\] (see, e.g., \cite[Corollary 3.4]{Walsh} or \cite[Theorem 6.8]{DKMNX}).
	Thanks to this and a quantitative form of the Kolmogorov continuity theorem
	(see, e.g., Proposition \ref{pr:KCT}), there exists a real number $c_0$ such that, uniformly
	for all $\tau\in[0\,,1]$,
	\begin{align*}
		\sup_{n\ge 1}\P\left( \left. T_{n+1}- T_n \le \tau\ \right|\, \cF_{T_n}\right)
			&\le \e^2\sup_{n\ge 1}\E\left( \sup_{x\in\mathbb{T}}\sup_{t\in[0,\tau]}
			\left|z^{(n)}(t\,,x)- z^{(n)}_0(x)\right|^2\right)\\
		&\le c_0 \lambda^2 \lip_{\sigma}^2
        		\sqrt{\tau}\e^{c_0\lambda^4 \lip_{\sigma}^4 \tau}, 
	\end{align*}
	a.s.. Because $c_0$ does not depend on $\tau\in[0\,,1]$, we may choose a 
	special $\tau=\tau(\lambda\,,\lip_\sigma,c_0)$ by setting
	\begin{equation}\label{tau:lambda}
		\tau :=\frac{\delta^2}{(\lambda\lip_\sigma)^4}\wedge 1,
	\end{equation}
	where $\delta=\delta(c_0)$ is the unique strictly-positive solution to
	$c_0\delta\exp(c_0\delta^2)=\frac12$. This yields
	\begin{equation}\label{PE}
		\P\left( \left. T_{n+1}- T_n >  \tau\ 
		\right|\, \cF_{T_n}\right) \ge \tfrac12
		\qquad\text{for all $n\ge 1$ a.s.},
	\end{equation}
	for the particular choice of $\tau$ that is furnished by \eqref{tau:lambda}.
	We now apply Lemma \ref{lem:LD} with
	\[
		J_n:=\1_{\{\omega\in\Omega:\,
		T_{n+1}(\omega)-T_n(\omega)> \tau \}},
	\]
	where $\tau$ is given by
	\eqref{tau:lambda} and  $q=\varepsilon=\nicefrac12$
	in order to deduce from \eqref{PE} that
	\[
		\P\left\{ \sum_{i=1}^n \1_{\{T_{i+1}-T_i> \tau \}} \le 
		\frac{n}{4}\right\} \le \e^{-n/16}
		\qquad\text{for all integers $n\ge1$}.
	\]
	Because
	\[
		T_n \ge \sum_{i=0}^{n-1}(T_{i+1}-T_i)\1_{\{T_{i+1}-T_i> \tau \}} 
		\ge \tau \sum_{i=1}^n \1_{\{T_{i+1}-T_i> \tau \}},
	\]
	we find that
	\begin{equation}\label{Ltail}
		\P\left\{ \inf_{x\in\mathbb{T}}\inf_{0\le t\le \tau n/4} u(t\,,x\,;\lambda) 
		\le \e^{-n}\underline{u}_0\right\}\le 
		\P\left\{ T_n \le \frac{\tau n}{4}\right\} \le \e^{-n/16}
		\qquad\text{for all $n\ge 1$}.
	\end{equation}
	Since
	\[
		\inf_{x\in\mathbb{T}}\inf_{\tau(n-1)/4 \le t \le \tau n/4} 
		\frac{u(t\,,x\,;\lambda)}{\exp(-n)}
		\le\e\inf_{x\in\mathbb{T}} \inf_{\tau(n-1)/4 \le t\le \tau n/4}\frac{u(t\,,x\,;\lambda)}{\exp\left(- 4t/\tau \right)},
	\]
	the Borel--Cantelli lemma implies the remaining half of Theorem \ref{th:decay:t}.
\end{proof}


\section{Proof of Theorem \ref{th:interm1}}
The proof of Theorem \ref{th:interm1} is based on  the proofs of  Theorem \ref{th:decay:t} and Theorem \ref{th:lambda:UB}.
\begin{proof}[Proof of Theorem \ref{th:interm1}]   
Let $t\geq 1$. Define $A_1(t)$ as was defined in \eqref{A}, then \eqref{P(A)} says
\begin{equation}\label{P(A1)}
 \P(A_1(t)) \geq 1- \e^{-\lambda^2 \rm{L}_\sigma^2 t/64}. 
 \end{equation}
We can also get  the following variation of \eqref{oh:boy}: 
\[
	\E\left(\sup_{x\in\mathbb{T}}
	\left| u(t+h\,,x\,;\lambda)\right|^k ;\ A_1(t) \right)
	\le \frac
	   {
	     C_1^k\e^{C_2\lambda^4 k^3 [(h/\varepsilon^2 )+\varepsilon t]}
	     \|u_0\|_{L^1}^{k(1-\varepsilon/2)}
		   \|u_0\|_{L^\infty}^{k\varepsilon/2}
	   }
		 {
		   h^{k/2} t^{k\varepsilon/2}
		 }
	\e^{-k\lambda^2{\rm L}_\sigma^2 t/16}
\]
for some constants $C_1, C_2>0$ which are independent of $k$ and $t$. We now choose 
\[ h:=\frac{1}{2} \quad \text{and}\quad \varepsilon:=\frac{\rm{L}_\sigma^2}{32C_2k^2\lambda^2}\wedge \varepsilon_0 \ \ \text{($\varepsilon_0$ is defined in Proposition \ref{pr:improved:x}}) \] to get that for some constant $\tilde{c}_k>0$ which only depends on $k$,
  \[
      \E\left(\sup_{x\in\mathbb{T}}
	    \left| u(t+1/2\,,x\,;\lambda)\right|^k ;\ A_1(t) \right)
	    \le \tilde{c}_{k} \e^{-k\lambda^2 \rm{L}^2_\sigma t/32}. 
	\]	
Replacing $t+1/2$ by $t$ above and redefining $A_1(t):=A_1(t-1/2)$, we get 
\begin{equation}\label{upper}
\E\left(\sup_{x\in\mathbb{T}}
	\left| u(t\,,x\,;\lambda)\right|^k ;\ A_1(t) \right)
	\le \tilde{c}_{k}\e^{k\lambda^2\rm{L}^2_\sigma/64} \,\e^{-k\lambda^2 \rm{L}^2_\sigma t/32}. 
\end{equation}
Let us now define 
\[ A_2(t):=\left\{\omega \in \Omega:\, \inf_{x\in\mathbb{T}}\inf_{0\le s\le t} u(s\,,x\,;\lambda)(\omega) 	\ge \e^{-4t/\tau}\underline{u}_0 \right\} \]
where $\tau=\tau(\lambda, \lip_\sigma, c_0)>0$ is the constant defined  on \eqref{tau:lambda} and $\underline{u}_0:=\inf_{y\in\mathbb{T}}u_0(y)$. By \eqref{Ltail}, we get 
\begin{equation}\label{P(A2)}
 \P(A_2(t)) \geq 1-\e^{-t/4\tau}.
\end{equation}
Let 
\[ B(t):=A_1(t)\cap A_2(t).\] 
By \eqref{P(A1)} and \eqref{P(A2)}, we have 
\[
    \begin{aligned}
       \P(B(t))&=\P(A_1(t)) +\P(A_2(t)) - \P(A_1(t) \cup A_2(t)) \\
       & \geq 1- \e^{-t/4\tau} -\e^{\lambda^2\rm{L}^2_\sigma / 128} \e^{-\lambda^2 \rm{L}_\sigma^2 t/64} \geq 1-b_1 \e^{-b_2},
    \end{aligned}
\]
where $b_1:=2\max\left\{ 1, \e^{\lambda^2\rm{L}^2_\sigma/128}\right\}$ and $b_2 := \min \left\{\frac{1}{4\tau}, \frac{\lambda^2\rm{L}^2_\sigma}{64} \right\}$.  Hence, there exists $c>0$ for some $t_0$ large, if $t\geq t_0$, 
\[ \P(B(t)) \geq 1-c\e^{-ct}\geq 1-c\e^{-ct_0}>0.\] 
This shows the first statement of Theorem \ref{th:interm1}. For the second statement, the upper bound comes from \eqref{upper} and the lower bound comes from the following:
\[
 \E\left(\inf_{x\in\mathbb{T}}
	\left| u(t\,,x\,;\lambda)\right|^k ;\ B(t) \right) \geq \e^{-4t/\tau}\underline{u}_0\, \P(B(t)) \geq \e^{-4t/\tau}\underline{u}_0 \left( 1-c\e^{-ct_0}\right),
\]
which completes the proof.

\end{proof}

\appendix


\section{A real-variable inequality}

We will have need for the following.

\begin{lemma}\label{lem:IJ}
	For all $\varepsilon\in(0\,,1)$, $\alpha\in[0\,,1)$,
	and $\beta\ge 1$,
	\[
		\sup_{t>0}
		\int_0^t\left(\frac{t}{s}\right)^{1-\varepsilon}
		\frac{\e^{-\beta(t-s)}}{(t-s)^\alpha}\,\d s 
		\le \frac{2\Gamma(1-\alpha)+ 1}{(1-\alpha)\varepsilon\beta^{1-\alpha}},
	\]
	where $\Gamma$ denotes Gamma function.
\end{lemma}

\begin{proof}
	By scaling, we might as well assume that $\beta=1$.
	Now, a change of variables yields
	\[
		\int_0^t\left(\frac{t}{s}\right)^{1-\varepsilon}
		\frac{\e^{-(t-s)}}{(t-s)^\alpha}\,\d s =
		t^{1-\alpha}\int_0^1\frac{\e^{- tr}}{r^\alpha(1-r)^{1-\varepsilon}}\,\d r,
	\]
	whenever $t>0$ and $0<\varepsilon<1$. If
	$t\le 1$, then we merely bound $t^{1-\alpha}$ and
	$\exp(- tr)$  by 1 in order to see that
	the preceding is at most ${\rm B}(\varepsilon\,,1-\alpha)$,
	where $\rm B$ is beta function.
	On the other hand, if $t>1$, then we change variables a few more times
	in order to see that
	\[
		\int_0^t\left(\frac{t}{s}\right)^{1-\varepsilon}
		\frac{\e^{-(t-s)}}{(t-s)^\alpha}\,\d s = \int_0^{t}
		\frac{\e^{-s}}{\left( 1 - \dfrac{s}{t}\right)^{1-\varepsilon}
		{s^\alpha}}\,\d s
		= \int_0^{t/2}(\,\cdots)\,\d s +
		\int_{ t/2}^{t}(\,\cdots)\,\d s,
	\]
	notation being obvious from context. Since
	$(1-[s/t])^{1-\varepsilon}\ge 2^{-1+\varepsilon}\ge\nicefrac12$
	for all $0<s< t/2$,
	\[
		\int_0^{ t/2}(\,\cdots)\,\d s \le 2\Gamma(1-\alpha).
	\]
	On the other hand,
	\[
		\int_{t/2}^{t}(\,\cdots)\,\d s \le
		\e^{-t/2}\int_{t/2}^{ t}\frac{\d s}{
		s^\alpha\left( 1 - \dfrac{s}{t}\right)^{1-\varepsilon}}
		= t^{1-\alpha}\,\e^{-t/2}\int_{1/2}^1\frac{\d r}{
		r^\alpha(1-r)^{1-\varepsilon}} \le
		{\rm B}(\varepsilon\,,1-\alpha),
	\]
	where we have used the elementary bound,
	$x^{1-\alpha}\e^{-x/2}\le 1$, valid for all $x\ge 0$.
	
	The preceding argument yields the inequality,
	\[
		\sup_{t>0} \int_0^t\left(\frac{t}{s}\right)^{1-\varepsilon}
		\frac{\e^{-(t-s)}}{(t-s)^\alpha}\,\d s 
		\le 2\Gamma(1-\alpha)+ 
		{\rm B}(\varepsilon\,,1-\alpha),
	\]
	valid for all $\varepsilon\in(0\,,1)$ and $\alpha\in[0\,,1)$.
	This implies the lemma since a change of variables yields
	\[
		{\rm B}(\varepsilon\,,1-\alpha) = \varepsilon^{-1}\int_0^\infty
		\e^{-y}\left( 1- \e^{-y/\varepsilon}\right)^{-\alpha}\d y
		\le \varepsilon^{-1}\int_0^\infty\e^{-y}\left( 1- \e^{-y}
		\right)^{-\alpha}\d y = \frac{1}{\varepsilon(1-\alpha)},
	\]
	for all $\varepsilon\in(0\,,1)$ and $\alpha\in[0\,,1)$, which is a well-known fact
	about the beta integral; see, Dragomir et al.
	\cite[(3.17)]{DAB} for a different proof of the latter inequality.
\end{proof}


\section{The heat kernel}
Recall $G$ and $p$ respectively from \eqref{eq:G} and \eqref{eq:p}.
\begin{lemma}\label{lem:c_l}
	For all $x,y\in[-1\,,1]$ and $t>0$,
	\[
		G_t(x-y) \le p_t(x\,,y) \le 2\max\left( \frac{1}{\sqrt t}\,,1\right).
	\]
\end{lemma}

\begin{remark}
	By Lemma \ref{lem:c_l},
	$\sup_{-1<x,y<1} p_t(x\,,y) \ge G_t(0) = (4\pi t)^{-1/2}$,
	pointwise. Also, $2\sup_y p_t(x\,,y) \ge \int_{-1}^1
	p_t(x\,,y) \,\d y = 1$, for all $t>0$ and $x\in\mathbb{T}$. Therefore,
	Lemma \ref{lem:c_l} has the following consequence:
	\[
		\frac14 \max\left( \frac{1}{\sqrt t}\,,1\right)\le
		\sup_{x,y\in[-1,1]} p_t(x\,,y) \le 2\max\left( \frac{1}{\sqrt t}\,,1\right),
	\]
	for all $t>0$.
\end{remark}

\begin{proof}[Proof of Lemma \ref{lem:c_l}]
	The lower bound is immediate; we establish the upper bound.
	
	Consider the summands in \eqref{eq:p} for $|n|\le 1$ and 
	$|n|\ge 2$ separately in order to see that 
	\[
		p_t(x\,,y) \le\frac{3}{\sqrt{4\pi t}} + \frac{1}{\sqrt{4\pi t}}
		\sum_{\substack{n\in\Z:\\|n|\ge2}}
		\exp\left( -\frac{(x-y-2n)^2}{4t}\right),
	\]
	for all $t>0$ and $x,y\in[-1\,,1]$. Since $(a-b)^2\ge \frac12 a^2-b^2$
	for all $a,b\in\R$, the preceding yields 
	$(x-y-2n)^2 \ge 2n^2 - (x-y)^2 \ge 2 (n^2-2)
	\ge n^2,$
	for all $x,y\in[-1\,,1]$ and integers $n$ with $|n|\ge 2$. Thus, we obtain
	the bound,
	\begin{equation}\label{peexp}
		p_t(x\,,y) \le\frac{1}{\sqrt{t}} + \frac{1}{\sqrt{\pi t}}
		\sum_{n=2}^\infty\exp\left( -\frac{n^2}{4t}\right),
	\end{equation}
	for all $t>0$ and $x,y\in[-1\,,1]$. In particular,
	\[
		\sup_{x,y\in[-1,1]}
		p_t(x\,,y) \le \frac{1}{\sqrt{t}}\left(1 + 
		\sum_{n=2}^\infty \e^{-n^2/4}\right)\le\frac{2}{\sqrt t},
	\]
	uniformly for $t\in(0\,,1]$. If $t>1$, then use
	\[
		\sum_{n=2}^\infty\exp\left( -\frac{n^2}{4t}\right)
		\le \int_0^\infty \exp\left( - \frac{z^2}{4t}\right)\d z
		= \sqrt{\pi t},
	\]
	in \eqref{peexp}, to see that
	\[
		\sup_{x,y\in[-1,1]} p_t(x\,,y) \le \frac{1}{\sqrt{t}} + 1,
	\]
	which is at most $2$.
\end{proof}

\begin{lemma}\label{lem:p-p:x:1}
	There exists a finite constant $C$ such that
	\[
		\int_{-1}^1\left| p_t(x\,,w)-p_t(y\,,w)\right|^2\d w \le
		C\frac{|x-y|}{t\wedge \sqrt t},
	\]
	uniformly for all $x,y\in[-1\,,1]$ and $t>0$.
\end{lemma}

\begin{proof}
	Choose and fix some $t>0$ and
	$x,y\in[-1\,,1]$.
	Without loss of generality, we assume that $x>y$.
	By the Chapman--Kolmogorov property, and thanks to the symmetry of
	$p_t$,
	\begin{align*}
		\int_{-1}^1\left| p_t(x\,,w)-p_t(y\,,w)\right|^2\d w
			&= p_{2t}(x\,,x) + p_{2t}(y\,,y) - 2p_{2t}(x\,,y)\\
		&= 2\sum_{n=-\infty}^\infty\left[ G_{2t}(2n) - G_{2t}(2n+x-y)\right].
	\end{align*}
	Because $|G_t'(a)| = |a|G_t(a)/2t$,
	\begin{equation}\label{eq:G:G}
		\left| G_t(2n) - G_t(2n+x-y)\right|
		\le \frac{1}{4\sqrt{\pi} t^{3/2}}\int_0^{x-y}|2n+a|\exp\left( -
		\frac{(2n+a)^2}{4t}\right)\d a,
	\end{equation}
	for all $n\in\Z$. In particular,
	\[
		\left| G_t(2n) - G_t(2n+x-y)\right|
		\le \frac{x-y}{4\sqrt{\pi} t^{3/2}}\sup_{r\ge0}\left[ r\exp\left( -
		\frac{r^2}{4t}\right)\right] \le \frac{x-y}{4\sqrt{\pi} t}.
	\]
	The preceding is useful when $|n|$ is not too large, say $|n|\le 2$. 
	On the other hand, if $|n|\ge 2$ and $0\le a\le 2$, then $3|n|\ge|2n+a|\ge
	2(|n|-1)\ge|n|$ . Therefore, \eqref{eq:G:G} implies that
	\[
		\left| G_t(2n) - G_t(2n+x-y)\right|
		\le \frac{3(x-y)}{4\sqrt{\pi} t^{3/2}}\cdot |n|\exp\left( -
		\frac{n^2}{4t}\right).
	\]
	We combine the preceding two displays to see that
	\[
		\int_{-1}^1\left| p_t(x\,,w)-p_t(y\,,w)\right|^2\d w
		\le \frac{3(x-y)}{2\sqrt{\pi} t} + \frac{3(x-y)}{2\sqrt{\pi} t^{3/2}}\cdot
		\sum_{|n|\ge 2}|n|\exp\left( - \frac{n^2}{4t}\right).
	\]
	The sum is at most $\int_{-\infty}^\infty|w|\exp(-w^2/(8t))\,\d w\propto t$,
	and this implies the result.
\end{proof}

\begin{lemma}\label{lem:p-p:x:2}
	For each $\delta\in(0\,,1)$
	there exists a finite constant $C$ such that
	\[
		\int_{-1}^1\left| p_t(x\,,w)-p_t(y\,,w)\right|^2\d w \le
		C\frac{|x-y|^\delta}{t^{(\delta+1)/2}\wedge t^{\delta/2}},
	\]
	uniformly for all $t>0$ and $x,y\in[-1\,,1]$.
\end{lemma}

\begin{proof}
	Since $(a+b)^2\le2a^2+2b^2$ for all $a,b\in\R$, the Chapman--Kolmogorov
	property yields the following for all $t>0$ and $x,y\in[-1\,,1]$:
	\[
		\int_{-1}^1\left| p_t(x\,,w)- p_t(y\,,w) \right|^2\,\d w
		\le 2p_{2t}(x\,,x)+2p_{2t}(y\,,y)\le 8\max\left(\frac{1}{\sqrt t},1\right);
	\]
	see Lemma \ref{lem:c_l}.  Therefore, Lemma \ref{lem:p-p:x:1} implies that
	we can find a finite constant $C$ such that, uniformly for all $x,y\in[-1\,,1]$
	and $t>0$,
	\[
		\int_{-1}^1\left| p_t(x\,,w)- p_t(y\,,w) \right|^2\,\d w
		\le C\frac{|x-y|}{\sqrt{t}}\max\left\{\frac{1}{\sqrt{t}},1\right\},
	\]
	The lemma follows since $\min(A\,,B)\le A^\delta B^{1-\delta}$
	for all $A,B\ge0$ and $\delta\in(0\,,1)$.
\end{proof}

\begin{lemma}\label{lem:P_th(x)-P_th(y)}
	There exists a finite constant $C$ such that,
	uniformly for all $\varepsilon\in(0\,,1)$,  $t>0$, $x,y\in[-1\,,1]$, and 
	$h\in L^1\cap L^\infty$,
	\[
		\left| (P_th)(x) - (P_th)(y) \right| \le C
		\max\left\{1\,, \frac{1}{\sqrt{t}}\right\}|x-y|^{\varepsilon/2} 
		\|h\|_{L^\infty}^{\varepsilon}\|h\|_{L^1}^{1-\varepsilon}.
	\]
\end{lemma}

\begin{proof}

First, we use Lemma \ref{lem:c_l} to get that 
	\[ 
	\left| (P_th)(x) - (P_th)(y) \right| \le 4\max\left\{1
	\,, \frac{1}{\sqrt{t}}\right\}\|h\|_{L^1} .
	\]
	We can also apply  the Cauchy--Schwarz inequality and then Lemma \ref{lem:p-p:x:1} 
	to obtain
	\begin{align*}
		\left| (P_th)(x)  - (P_th)(y) \right|^2 &\le 
			\left(\int_{-1}^1 |p_t(x\,,w)-p_t(y\,,w)| \cdot |h(w)|\,\d w\right)^2\\
		&\le 2\|h\|_{L^\infty}^2\int_{-1}^1|p_t(x\,,w)-p_t(y\,,w)| ^2\,\d w\\
		&\le C \|h\|_{L^\infty}^2 \frac{|x-y| }{t\wedge \sqrt{t}}\\
		&\le C \|h\|_{L^\infty}^2|x-y| \max\left\{1\,, \frac{1}{t}\right\}.		
	\end{align*}
	Now the lemma follows since $\min(A\,,B)\le A^\varepsilon B^{1-\varepsilon}$
	for all $A,B\ge0$ and $\varepsilon\in(0\,,1)$.
\end{proof}

\begin{lemma}\label{lem:p-p:t:1}
	For each $t,\delta>0$,
	\[
		\sup_{x\in\mathbb{T}}
		\int_{-1}^1\left| p_{t+\delta}(x\,,w)-p_t(x\,,w)\right|^2\d w \le
		\sqrt{\frac{\pi}{2t}}\min\left(1\,,\frac{\delta}{4t}\right).
	\]
\end{lemma}

\begin{proof}
	Choose and fix some $t,\delta>0$ and $x\in\mathbb{T}$.
	By the Chapman--Kolmogorov property, and thanks to the symmetry of
	$p_t$,
	\begin{align*}
		\int_{-1}^1\left| p_{t+\delta}(x\,,w)-p_t(x\,,w)\right|^2\d w
			&= p_{2(t+\delta)}(x\,,x) + p_{2t}(x\,,x) - 2p_{2t+\delta}(x\,,x)\\
		&= \sum_{n=-\infty}^\infty \left[ G_{2t+2\delta}(2n) +
			G_{2t}(2n)-2G_{2t+\delta}(2n)\right].
	\end{align*}
	Because the Fourier transform of $F(x):=G_\tau(2x)$ is 
	$\widehat{F}(z)=\frac12\widehat{G}_\tau(z/2)=\frac12\exp(-\tau z^2/4),$
	the Poisson summation formula \cite[page 161]{Katznelson} implies that
	\[
		\int_{-1}^1\left| p_{t+\delta}(x\,,w)-p_t(x\,,w)\right|^2\d w
		=\tfrac12 \sum_{n=-\infty}^\infty
		\e^{-tn^2/2}\left(\e^{-\delta n^2/2}+ 1- 2\e^{-\delta n^2/4}\right)
		\le  \sum_{n=1}^\infty
		\e^{-tn^2/2}\min\left(1\,, \frac{\delta n^2}{4}\right),
	\]
	uniformly for all $t>0$ and $\delta\in(0\,,1)$.  This readily implies
	the result since
	\[
		\sum_{n=1}^\infty \e^{-tn^2/2} \le \int_0^\infty \e^{-tx^2/2}\,\d x
		=\sqrt{\frac{\pi}{2t}},
	\]
	and
	\[
		\delta \sum_{n=1}^\infty \e^{-tn^2/2} n^2 \le\delta\int_0^\infty
		x^2\e^{-tx^2/2}\,\d x=\frac{\delta}{t}\sqrt{\frac{\pi}{2t}},
	\]
	for all $t>0$ and $\delta\in(0\,,1)$.
\end{proof}

\begin{lemma}\label{lem:P_th(x)-P_sh(x)}
	For every $t,\delta>0$, $\varepsilon\in(0\,,1)$,  and $h\in L^1\cap L^\infty$,
	\[
		\sup_{x\in\mathbb{T}}
		\left| (P_{t+\delta}h)(x) - (P_th)(x) \right| \le 4\left(1+\frac{1}{\sqrt{t}} \right) 
		\min\left(1\,,\left[\frac{\delta}{4t}\right]^{\varepsilon/2}\right)
		\cdot
		\|h\|_{L^\infty}^{\varepsilon}\|h\|_{L^1}^{1-\varepsilon}.
	\]
\end{lemma}

\begin{proof}

We first use Lemma \ref{lem:c_l}, as we did in the proof of Lemma \ref{lem:P_th(x)-P_th(y)}, to get that
\[ \left| (P_{t+\delta}h)(x) - (P_th)(x) \right| \le 4\max\left(1, \frac{1}{\sqrt{t}} \right) \|h\|_{L^1}.\]
We now apply first the Cauchy--Schwarz inequality and then Lemma \ref{lem:p-p:t:1} in order to deduce
	\begin{align*}
		\left| (P_{t+\delta}h)(x)  - (P_th)(x) \right|^2 
			&\le 2\|h\|_{L^\infty}^{2}
			\int_{-1}^1|p_{t+\delta}(x\,,w)-p_t(x\,,w)| ^2\,\d w \\
		&\le  2\|h\|_{L^\infty}^{2} \sqrt{\frac{\pi}{2t}}\min\left(1\,,\frac{\delta}{4t}\right)\\
		&\le  4\max\left(1\,, \frac{1}{t} \right) \|h\|_{L^\infty}^{2}
			\min\left(1\,,\frac{\delta}{4t}\right).
	\end{align*}
	Now the lemma follows since  $\max\{A\,, B\} \leq A+B$ and
	$\min(A\,,B)\le A^\varepsilon B^{1-\varepsilon}$
	for all $A,B\ge0$ and $\varepsilon\in(0\,,1)$.
\end{proof}


\bigskip

\begin{small}
\noindent {\bf D. Khoshnevisan.} Department of Mathematics, University of Utah,
	Salt Lake City, UT 84112-0090, USA,
	\texttt{davar@math.utah.edu}\\[.2cm]
\noindent {\bf Kunwoo Kim.} Department of  Mathematics, Pohang University of Science and Technology (POSTECH), Pohang, Gyeongbuk, Korea 37673,
	\texttt{kunwoo@postech.ac.kr}\\[.2cm]
\noindent\textbf{Carl Mueller.} Department of Mathematics, University of Rochester,
	Rochester, NY 14627, USA,\\
	\texttt{http://www.math.rochester.edu/people/faculty/cmlr}\\[.2cm]
\noindent\textbf{Shang-Yuan Shiu.} Department of Mathematics, 
	National Central University,
	No.\ 300, Jhongda Rd., Jhongli City, Taoyuan County 32001, Taiwan (R.O.C.)
	\texttt{shiu@math.ncu.edu.tw}
\end{small}


\begin{thebibliography}{99}
%
\bibitem{BC} Bertini, Lorenzo, and Nicoletta Cancrini,
	\newblock The stochastic heat equation: Feynman-Kac formula and intermittency,
	\newblock {\it J. Statist.\ Phys.}\ 
	\newblock (1995)
	\newblock {\bf 78}{\it (5/6)}
	\newblock 1377--1401.
	
	
\bibitem{CM} Carmona, Ren\'e A.,
	and S. A. Molchanov,  
	\newblock Parabolic Anderson problem and intermittency,
	\newblock {\it Mem.\ Amer.\ Math.\ Soc.}\
	\newblock (1994)
	\newblock {\bf 108}{\it (518)}.

\bibitem{ChenKim} Chen, Le and Kunwoo Kim,
  	\newblock On comparison principle and strict positivity of solutions to the nonlinear stochastic fractional heat equations,
	\newblock {\it Ann.\ I.\ H.\ Poincar\'e}
	\newblock (2017)
	\newblock {\bf 53} no. 1,
	\newblock 358--388.


\bibitem{CJK} Conus, Daniel, Mathew Joseph, and Davar Khoshnevisan,
	\newblock Correlation-length bounds, and estimates for intermittent islands
		in parabolic SPDEs,
	\newblock {\it Electr.\ J. Probab.}\
	\newblock (2012)
	\newblock {\bf 17} no.\ 102
	\newblock 1--15.

\bibitem{CMS2002} Cranston, Michael, Thomas S. Mountford, and Tokuzo Shiga,
	\newblock Lyapunov exponents for the parabolic Anderson model,
	\newblock {\it Acta Math. Univ. Comenian. (N.S.)}
	\newblock {\bf 71}{\it (2)}
	\newblock (2002)
	\newblock 321--355.

\bibitem{DZ} Da Prato, Giuseppe and Zabczyk, Jerzy,
	\newblock {\it Stochastic equations in infinite dimensions},
	\newblock Encyclopedia of Mathematics and its Applications,
	\newblock Volume 44,
	\newblock Cambridge University Press, Cambridge,
	\newblock 1992.

\bibitem{Dalang} Dalang, Robert C.,
	\newblock Extending the martingale measure stochastic integral with
		applications to spatially homogeneous s.p.d.e.'s,
	\newblock {\it Electron. J. Probab.} 
	\newblock (1999)
	\newblock {\bf 4}{\it (6)} 29 pp.
	[Corrections: {\it Electron J. Probab.} {\bf 6}{\it (6)} (2001) 5 pp.]

\bibitem{DKMNX} Dalang, Robert, Khoshnevisan, Davar, Mueller, Carl, Nualart, David and Xiao, Yimin.
   \newblock A minicourse on stochastic partial differential equations. 
   \newblock Held at the University of Utah, Salt Lake City, UT, May 8–19, 2006.
	 \newblock Edited by Khoshnevisan and Firas Rassoul-Agha.
	 \newblock Springer-Verlag, Berlin,
	 \newblock 2009.

	
	

	
\bibitem{DAB} Dragomir, S. S., R. P. Agarwal, and N. S. Barnett,
	\newblock Inequalities for beta and gamma functions via some
		classical and new integral inequalities,
	\newblock {\it J. Ineq.\ \&\ Appl.}\
	\newblock (2000) Vol.\ 5, pp.\ 103--165.
	
	
	
	
\bibitem{FK} Foondun, Mohammud and Davar Khoshnevisan,
	\newblock Intermittence and nonlinear stochastic partial differential equations,
	\newblock {\it Electr.\ J. Probab.}\
	\newblock (2009)
	\newblock {\bf 14}{\it (21)}
	\newblock 548--568.

\bibitem{FN} Fonodun, Mohammud and Eulalia Nualart,
	\newblock On the behaviour of stochastic heat equations on bounded domains,
	\newblock {\it ALEA Lat.\ Amer.\ J. Probab.\ Math.\ Stat.}\ 
	\newblock (2015)
	\newblock {\bf 12}{\it (2)}  551--571.
	 


\bibitem{GD} Gibbon, J. D. and Charles Doering,
	\newblock Intermittency and regularity issues in 3-D Navier-Stokes turbulence,
	\newblock {\it Arch.\ Rational Mech.\ Anal.}\
	\newblock (2005) 
	\newblock {\bf 177} 
	\newblock 115--150.
	
\bibitem{GT} Gibbon, J. D. and E. S. Titti,
	\newblock Cluster formation in complex multi-scale systems,
	\newblock {\it Proc.\ Royal Soc.\ A}  
	\newblock (2005)
	\newblock {\bf 461}
	\newblock 3089--3097.
		
\bibitem{Katznelson} Katznelson, Yitzhak,
	\newblock {\it An Introduction to Harmonic Analysis},
	\newblock Third Edition,
	\newblock Cambridge Mathematical Library,
	\newblock Cambridge University Press, Cambridge,
	\newblock 2004.

\bibitem{cbms} Khoshnevisan, Davar,
	\newblock {\it Analysis of Stochastic Partial Differential Equations},
	\newblock NSF--CBMS Regional Conf.\ Series in Math.\ {\bf 119}
	\newblock (2014)
	\newblock American Math.\ Soc., Providence, RI.

\bibitem{KK1} Khoshnevisan, Davar and Kunwoo Kim,
	\newblock Nonlinear noise excitation of intermittent stochastic PDEs and the topology of LCA groups,
	\newblock {\it Ann.\ Probab.}\
	\newblock (2015),
	\newblock {\bf 43}{\it (4)},
	\newblock 1944--1991. 

\bibitem{KK2} Khoshnevisan, Davar and Kunwoo Kim,
	\newblock Non-linear noise excitation and intermittency under high disorder,
	\newblock {\it Proc. \ Amer. \ Math.\ Soc.}\
	\newblock (2015),
	\newblock {\bf 143}{\it (9)},
	\newblock 4073--4083. 



	
\bibitem{KKX1} Khoshnevisan, Davar, Kunwoo Kim, and Yimin Xiao,
	\newblock Intermittency and multifractality: a case study via parabolic stochastic PDEs,
	\newblock {\it Ann.\ Probab.}\
	\newblock (2017),
	\newblock {\bf 45}{\it (6A)},
	\newblock 3697--3751. 

\bibitem{KKX2} Khoshnevisan, Davar, Kunwoo Kim, and Yimin Xiao,
	\newblock A macroscopic multifractal analysis of parabolic stochastic PDEs,
	\newblock {\it Comm.\ Math.\ Phys.}\
	\newblock (2018),
	\newblock {\bf 360}{\it (1)},
	\newblock 307--346. 
	
	
\bibitem{KRS} Khoshnevisan, Davar, P\'al R\'ev\'esz, and Zhan Shi,
  	\newblock On the explosion of the local times along lines of Brownian sheet,
	\newblock {\it Ann.\ I.\ H.\ Poincar\'e}
	\newblock (2004)
	\newblock {\bf 40}
	\newblock 1--24.

%
%
\bibitem{Mueller1} Mueller, Carl,
	\newblock On the support of solutions to the heat equation with noise.
	\newblock {\it Stoch.\ \&\ Stoch.\ Rep.}\ {\bf 37}{\it (4)} (1991) 225--245.
%
\bibitem{MN} Mueller, Carl and David Nualart,
	\newblock Regularity of the density for the stochastic heat equation.
	\newblock {\it Electr.\ J. Probab.}\ {\bf 13}{\it (74)} (2008), 2248--2258.
%
	
\bibitem{Nualart} Nualart, Eulalia,
	\newblock Moment bounds for some fractional stochastic heat equations on the ball,
	\newblock {\it Electron. Commun. Probab.}
	\newblock (2018)
	\newblock {\bf 23} {\it(41)}
	\newblock 1--12.
	
	
	
\bibitem{RY} Revuz, Daniel and Yor, Marc,
	\newblock {it Continuous martingales and {B}rownian motion},
	\newblock Grundlehren der Mathematischen Wissenschaften [Fundamental
              Principles of Mathematical Sciences],
	\newblock Volume 293,
	\newblock Third Edition, 
	\newblock Springer-Verlag, Berlin,
	\newblock 1999.

\bibitem{Shiga} Shiga, Tokuzo,
	\newblock Two contrasting properties of solutions for one-dimensional
              stochastic partial differential equations,
	\newblock {\it Canad. J. Math.}
	\newblock (1994)
	\newblock {\bf 46} {\it(2)}
	\newblock 415--437.

\bibitem{Walsh} Walsh, John B.
	\newblock {\it An Introduction to Stochastic Partial Differential Equations}.
	\newblock In:  \'Ecole d'\'et\'e de probabilit\'es de
	Saint-Flour, XIV---1984, 265--439.
	Lecture Notes in Math.\ {\bf 1180} Springer, Berlin, 1986.
	
\bibitem{Xie} Xie, Bin,
	\newblock Some effects of the noise intensity upon non-linear stochastic heat equations on $[0,1]$,
	\newblock {\it Stochastic Process. Appl.}
	\newblock (2016)
	\newblock {\bf 126}{\it(4)}
	\newblock 1184--1205.

	
\bibitem{ZMRS1988} Zel'dovich, Ta. B., S. A. Molchanov, A. A. Ruzmaikin,
	and D. D. Sokolov,
	\newblock Intermittency, diffusion, and generation in
	 	a nonstationary random medium,
	 \newblock {\it Sov.\ Sci.\ Rev.\ C. Math.\ Phys.}\
	 \newblock (1988)
	 \newblock {\bf 7}
	 \newblock 1--110.

\bibitem{ZMRS1985} Zel'dovich, Ta. B., S. A. Molchanov, A. A. Ruzmaikin,
	and D. D. Sokolov,
	\newblock Intermittency of passive fields in random media,
	\newblock {\it J. Experimental and Theoretical Physics},
	\newblock (1985)
	\newblock 2061--2072. (In Russian)
	
\bibitem{ZMS} Zel'dovich, Ta. B., A. A. Ruzmaikin,and D. D. Sokolov,
	\newblock {\it Almighty Chance},
	\newblock World Scientific Lecture Notes in Physics,
	\newblock Singapore, 1990.

\bibitem{ZTPS} Zimmerman, M. G., R. Toral O. Prio, and M. San Miguel,
	\newblock Stochastic spatiotemporal intermittency and noise-induced
		transition to an absorbing phase,
	\newblock {\it Phys.\ Rev.\ Lett.}\
	\newblock (2000)
	\newblock {\bf 85}{\it (17)} 3612--2615.
\end{thebibliography}
\end{document}